\documentclass[bib]{ba}

\usepackage{graphicx,hyperref}
\usepackage{bayes}

\RequirePackage[OT1]{fontenc}
\RequirePackage{amsthm,amsmath}

\usepackage{array}
\usepackage{dsfont}
\usepackage{graphicx}
\usepackage{mathrsfs}
\usepackage{amssymb}
\usepackage{amsbsy}
\usepackage{amsfonts}
\usepackage{latexsym}

\newcommand{\vir}[1]{``#1''}
\newcommand{\norm}[1]{\left\Vert#1\right\Vert}
\newcommand{\abs}[1]{\left\vert#1\right\vert}

\newcommand{\pg}[1]{\left\{#1\right\}}
\newcommand{\pq}[1]{\left[#1\right]}
\newcommand{\pt}[1]{\left(#1\right)}

\newcommand{\F}{\mathscr{F}}

\numberwithin{equation}{section}
\theoremstyle{plain}
\newtheorem{thm}{Theorem}[section]

\newtheorem{rmk}{Remark}[section]
\newtheorem{cor}{Corollary}[section]
\newtheorem{lem}{Lemma}[section]

\newtheorem{example}{\sc{Example}}[section]

\def\AD#1{\href{#1}{#1}}

\begin{document}


\inserttype[]{article}
\renewcommand{\thefootnote}{\fnsymbol{footnote}}

\author{Catia Scricciolo}{
 \fnms{Catia}
 \snm{Scricciolo}
 \footnotemark[1]\ead{catia.scricciolo@unibocconi.it}
}

\title[Adaptive Bayesian density estimation using mixture models]
{Adaptive Bayesian density estimation
using Pitman-Yor or normalized inverse-Gaussian process kernel mixtures}

\maketitle

\footnotetext[1]{
 Bocconi University
 \href{catia.scricciolo@unibocconi.it}{catia.scricciolo@unibocconi.it}
}
\renewcommand{\thefootnote}{\arabic{footnote}}

\begin{abstract}
We consider Bayesian nonparametric density estimation
using a Pitman-Yor or a normalized inverse-Gaussian process kernel mixture
as the prior distribution for a density. The procedure is studied from a frequentist perspective.
Using the stick-breaking representation of the Pitman-Yor process
or the expression of the finite-dimensional distributions
for the normalized-inverse Gaussian process, we prove that, when the data are replicates from
a density with Sobolev or analytic smoothness, the posterior distribution concentrates on
shrinking $L^p$-norm balls around the sampling density at a minimax-optimal rate,
up to a logarithmic factor. The resulting hierarchical Bayes procedure,
with a fixed prior, is thus shown to be adaptive
to the regularity of the sampling density.

\end{abstract}


\section{Introduction}
Consider the problem of estimating a univariate density
from independent and identically distributed (i.i.d.) observations taking a Bayesian
nonparametric approach. A prior is defined on a metric
space of probability measures with Lebesgue density and
a summary of the posterior, typically the posterior expected density,
can be employed as an estimator. Since the seminal articles of Ferguson~\cite{F} and Lo~\cite{L},
the idea of constructing priors on spaces of densities by convolving a fixed kernel
with a random distribution has been successfully exploited in density estimation.
A kernel mixture may provide an efficient approximation scheme,
possibly resulting in a minimax-optimal (up to a logarithmic factor) speed of concentration
for the posterior on shrinking balls around the sampling density.

Recent literature on Bayesian density estimation has mainly focussed on
posterior contraction rates relative to the Hellinger or the $L^1$-distance,
using a Dirichlet process mixture of (generalized) normals.
Ghosal and van der Vaart~\cite{GvdV01} found
a nearly parametric rate for estimating \emph{supersmooth} densities that are
themselves mixtures of normals, while Shen and Ghosal~\cite{SG},
extending the result of Kruijer \emph{et al.}~\cite{KRvdV10},
have proved that fully rate-adaptive multivariate density estimation
over Hölder regularity scales can be
performed using \emph{infinite} Dirichlet mixtures of Gaussians,
without any bandwidth shrinkage in the prior for the scale.

Even if much progress has been done during the last decade in understanding
frequentist asymptotic properties of kernel mixture models for Bayesian density estimation,
there seems to be a lack of results concerning adaptive estimation of ordinary and
infinitely smooth densities with respect to more general loss-functions than the Hellinger distance,
using other processes, apart from the Dirichlet process, as priors for the mixing distribution.
In this article, we investigate the question of how to complement and generalize
existing results on posterior contraction rates by considering
adaptive estimation over Sobolev or analytic density functional classes
using the Pitman-Yor or the normalized inverse-Gaussian process
as priors for the mixing distribution of general kernel mixtures.

The main results describe recovery rates for smooth densities,
where smoothness is measured through a scale of integrated tail bounds on the Fourier transform
of the density. For analytic densities a nearly parametric
rate arises under various priors which may
possibly affect only the power of the logarithm term, wherein the
characteristic exponent of the Fourier transform is automatically recovered.
Such a fast rate is roughly explainable from the fact that
spaces of analytic functions are only slightly bigger than finite-dimensional spaces
in terms of metric entropy. Besides in the prior distributions considered, the novelty of the work
is in the use of various and stronger norms to measure recovery rates, namely, the full scale of $L^p$-norms.
That a large class of Bayesian procedures are capable of such recovery is
established here for the first time and is encouraging to these methods.

Recovery rates for densities in Sobolev classes
are found to be minimax-optimal (up to a logarithmic factor) only under the Dirichlet or
the normalized-inverse Gaussian process for $L^p$-norms with $p\in[1,\,2]$, whereas they
deteriorate by a genuine power of $n$ as $p$ increases beyond $2$.
Slower rates are also found when endowing the mixing distribution with a Pitman-Yor process having strictly
positive discount parameter because small balls do not receive
enough prior mass. We currently have no proof that posterior contraction rates
are indeed sub-optimal under a Pitman-Yor process prior, but believe they
cannot be improved when the discount parameter is strictly positive.


Such results are of interest for a variety of reasons:
they may constitute a first step, beyond the Dirichlet process,
towards the study of posterior contraction rates for more involved process priors
recently proposed in the literature. Also, they provide an indication on the performance of
Bayes' procedures for adaptive estimation over functional classes
extensively considered in the frequentist literature on
nonparametric curve estimation.

The main challenge in proving the adaptation result for
the infinitely smooth case rests in finding a finite
mixing distribution, with a sufficiently restricted number of support
points, such that the corresponding Gaussian mixture approximates the sampling
density, in Kullback-Leibler divergence, with an error of the correct order.
Such a finitely supported mixing distribution may be found by matching the
moments of an \emph{ad hoc} constructed mixing density, for which, however, the method used by
Kruijer \emph{et al.}~\cite{KRvdV10} does not seem to be well-suited because of the infinite degree of smoothness of the true density. There are limitations implicitly coming from the kernel which are by-passed using superkernels, whose usefulness in density estimation has been pointed out by, among others, Devroye~\cite{Devroye92}. The crux and a main contribution of this article is the development of an approximation result for analytic densities with exponentially decaying Fourier transforms, cf. Lemma~\ref{approxlem}. We believe this result can be of autonomous interest as well and possibly exploited by frequentist methods in adaptive density estimation for clustering with Gaussian mixtures along the lines of Maugis and Michel~\cite{MM11}.

When assessing posterior rates, a major difficulty is the evaluation
of the prior concentration rate, calculated bounding below the prior probability of Kullback-Leibler type neighbourhoods by the prior probability of an $L^1$-ball of the right dimension.
For the normalized inverse-Gaussian process,
the expression of the finite-dimensional distributions is
used to estimate the probability of an $L^1$-ball as for the Dirichlet process.
For the Pitman-Yor process, instead, we exploit the stick-breaking representation to
obtain lower bounds on the probabilities of $L^1$-balls of the mixing weights
and locations. We expect this technique can be applied to other stick-breaking processes.

The exposition is focussed on density estimation, but
other statistical settings are implicitly covered: for example,
fixed design linear regression with unknown error distribution, as described in Ghosal and van der Vaart~\cite{GvdV071}, pages~205--206. Extension of these results to a multivariate setting seems imminent
along the lines of Shen and Ghosal~\cite{SG} and is not pursued here.

The organization of the article is as follows. In Section~\ref{first}, we fix the notation
and review preliminary definitions. In Section~\ref{sec:second},
we state results on posterior rates for general kernel mixtures
highlighting the connection with posterior recovery rates for mixing distributions.
The main results are reported in Section~\ref{sec:third}, where
after investigating the achievability of the error rate $1/\sqrt{n}$, up to a logarithmic factor,
for \emph{supersmooth} densities that admit a kernel mixture representation,
we focus on adaptive estimation of densities with
analytic or Sobolev smoothness using infinite Gaussian mixtures.
Prior estimates are given in Section~\ref{sec:estimates}.
Section~\ref{sec:ana} reports the proof of the
theorem on adaptive estimation of analytic densities.
Auxiliary results are deferred to the Appendix in Section~\ref{App}.

\subsection{Notation} Integrals where no limits are
written are to be taken over the entire real line. We write \vir{$\lesssim$} and
\vir{$\gtrsim$} for inequalities valid up to a constant multiple which is universal or
inessential for our purposes. For real numbers $a$ and $b$, we denote by $a\wedge b$ their minimum and by $a\vee b$ their maximum. For any real valued function $f$, we denote by $f^+$ its non-negative part $f1_{\{f\geq0\}}$.
We use the same symbol $F$ to denote the
distribution function and the corresponding probability measure.


\section{Model description}\label{first}
The model is
a \emph{location} mixture
$f_{F,\,\sigma}(\cdot):=(F\ast K_\sigma)(\cdot)=\int\sigma^{-1}K((\cdot-\theta)/\sigma)\,\mathrm{d}F(\theta)$,
where $K$ denotes the kernel density, $\sigma$ the scale parameter and $F$ the mixing distribution.
Kernels herein considered are characterized via a condition on the Fourier transform.
For finite constants $\rho,\,r,\,L>0$, let $\mathcal{A}^{\rho,\,r,\,L}(\mathbb{R})$ be the class of
densities on $\mathbb{R}$ with Fourier transform $\hat{f}(t):=\int e^{itx}f(x)\,\mathrm{d}x$, $t\in\mathbb{R}$,
satisfying
\begin{equation}\label{expdecr}
I^{\rho,\,r}(f):=\int e^{2(\rho|t|)^r}|\hat{f}(t)|^2\,\mathrm{d}t\leq 2\pi L.
\end{equation}
In symbols, $\mathcal{A}^{\rho,\,r,\,L}(\mathbb{R}):=\{f:\mathbb{R}\rightarrow\mathbb{R}^+|\,\,\|f\|_1=1,\,I^{\rho,\,r}(f)\leq 2\pi L\}$.
Condition (\ref{expdecr}) implies that the behaviour of $|\hat{f}|$ is
described by $e^{-(\rho|t|)^r}$ as $|t|\rightarrow\infty$.
Densities with Fourier transform satisfying (\ref{expdecr})
are \emph{infinitely} differentiable on $\mathbb{R}$, see, \emph{e.g.}, Theorem~11.6.2.
in Kawata~\cite{K72}, pages~438--439, and \vir{increasingly smooth} as
$\rho$ or $r$ increases. Also, they are bounded, $
\|f\|_\infty\leq(2\pi)^{-1}\int|\hat{f}(t)|\,\mathrm{d}t
\leq L+\pi^{-1}C(\rho,\,r)<\infty$,
where $C(\rho,\,r):=\int_0^\infty e^{-2(\rho t)^r}\,\mathrm{d}t=(2\rho^r)^{-1/r}\Gamma(1+1/r)$,
cf. Lemma~1 in Butucea and Tsybakov~\cite{BTI08}, page~35. Densities in classes
$\mathcal{A}^{\rho,\,r,\,L}(\mathbb{R})$ are called \emph{supersmooth}.
They form a larger class than that of
analytic densities, including important examples like Gaussian,
Cauchy, symmetric stable laws, Student's-$t$, distributions
with characteristic function vanishing outside a compact, as well as their mixtures and convolutions.
\begin{example}
\emph{\emph{Symmetric stable laws}, which have characteristic function of the form
$e^{-(\rho|t|)^r}$, $t\in\mathbb{R}$, for some $\rho>0$
and $0<r\leq2$, are supersmooth.
Cauchy laws $\textrm{Cauchy}(0,\,\sigma)$
are stable with $r=1$ and $\rho=\sigma$.
Normal laws $\textrm{N}(0,\,\sigma^2)$ are stable with $r=2$
and $\rho=\sigma/\sqrt{2}$.}
\end{example}
\begin{example}
\emph{\emph{Student's-$t$ distribution} with $\nu>0$ degrees of freedom
has characteristic function verifying \eqref{expdecr} for $r=1$: $\widehat{f}_{{t}_\nu}(t)\cong\sqrt{\pi}[\Gamma(\nu/2)2^{(\nu-1)/2}]^{-1}
(\sqrt{\nu}|t|)^{(\nu-1)/2}e^{-\sqrt{\nu}|t|}$ as $|t|\rightarrow\infty$,
see formula (4.8) in Hurst~\cite{H95}, page~5.}
\end{example}
\begin{example}\label{Ex:4}
\emph{\emph{Densities with characteristic function vanishing outside
a symmetric convex compact set} are supersmooth. Let $\Sigma_\Lambda$ be
the class of densities with characteristic function equal to zero outside
a symmetric convex compact set $\Lambda$ in $\mathbb{R}^k$, $k\geq1$.
For $k=1$, let $\Lambda=[-T,\,T]$ with $0<T<\infty$.
For any $f\in\Sigma_\Lambda$, it is $f\in\mathcal{A}^{\rho,\,r,\,L}(\mathbb{R})$
for every $\rho,\,r>0$ and $L\geq\pi^{-1}T e^{2(\rho T)^r}$.
The \emph{Fej\'er-de la Vall\'ee-Poussin} density
$f(x)=(2\pi)^{-1}[(x/2)^{-1}\sin(x/2)]^2$, $x\in\mathbb{R}$,
having $\hat{f}(t)=(1-|t|)^+$, $t\in\mathbb{R}$, is the typical example of
density in $\Sigma_\Lambda$, with $\Lambda=[-1,\,1]$.}
\end{example}
Classes of densities as in Example~\ref{Ex:4} are such that, even if
infinite-dimensional, nevertheless, for $p\geq2$,
$\inf_{f_n}
\sup_{f\in \Sigma_\Lambda}\operatorname{E}_f^n[\|f_n-f\|_p^s]\leq c_s n^{-s/2}$.
Moreover, for $p=s=2$, the precise asymptotic bound
$\lim_{n\rightarrow\infty}n
\inf_{f_n}
\sup_{f\in \Sigma_\Lambda}\operatorname{E}_f^n[\|f_n-f\|_2^2]=\textrm{meas}(\Lambda)/(2\pi)^k$ holds,
see Hasminskii and Ibragimov~\cite{HI1990}, page~1008, and the references therein.
The \emph{almost parametric} rate $(\log n)/n$ is achievable for densities with characteristic function
decreasing exponentially fast, see Watson and Leadbetter~\cite{WL63}. This rate was proved to be optimal in the minimax sense by Ibragimov and Hasminskii~\cite{IH83}.
Starting from this work, functional classes related to $\mathcal{A}^{\rho,\,r,\,L}(\mathbb{R})$
have been considered by many authors in frequentist nonparametric curve estimation. Just to mention a few,
Golubev and Levit~\cite{GL96} constructed asymptotically efficient
estimators of the density and its derivatives;
Golubev \emph{et al.}~\cite{GLT96} investigated nonparametric regression estimation;
Guerre and Tsybakov~\cite{GT98} studied estimation of the unknown signal in the Gaussian white noise model;
Butucea and Tsybakov~\cite{BTI08} considered adaptive density estimation in deconvolution problems.
Adaptive density or regression function estimation over classes
$\mathcal{A}^{\rho,\,r,\,L}(\mathbb{R})$
has so far hardly been studied from a Bayesian perspective, except for the recent works of
van der Vaart and van Zanten~\cite{vdV+vZ}, who use
a Gaussian random field with an inverse-gamma bandwidth, and of
de Jonge and van Zanten~\cite{deJvZ}, who use finite kernel mixture
priors with Gaussian mixing weights. The problem with the use of
finite mixtures is the choice of the number of components,
while updating it in a fully Bayesian way is computationally intensive.
Mixture models admitting an infinite discrete representation,
like the Dirichlet or more general stick-breaking processes, avoid fixing a truncation
level.
The focus of this work is on the capability of
general kernel mixture priors
to adapt posterior contraction rates to Sobolev or analytic smoothness of the sampling
density, 
without using knowledge about the regularity of $f_0$ in the definition.

Given the model $f_{F,\,\sigma}$, a prior is constructed on the space of Lebesgue univariate densities by
putting priors on the mixing distribution $F$ and the scale $\sigma$.
Let $\Pi$ denote the prior for $F$. The scale is assumed to be distributed, independently
of $F$, according to $G$ on $(0,\,\infty)$.
The overall prior $\Pi\times G$ on $\mathscr{M}(\Theta)\times(0,\,\infty)$,
where $\mathscr{M}(\Theta)$ stands for the set of all probability measures on $\Theta\subseteq\mathbb{R}$,
induces a prior on $\F:=\{f_{F,\,\sigma}:\,(F,\,\sigma)\in\mathscr{M}(\Theta)\times(0,\,\infty)\}$, which is
equipped with an $L^p$-metric $\|f-g\|_p:=(\int|f-g|^p\,\mathrm{d}\lambda)^{1/p}$, $p\in[1,\,\infty)$, where $\lambda$ denotes Lebesgue measure on $\mathbb{R}$, or with the sup-norm metric $\|f-g\|_\infty:=\sup_{x\in\mathbb{R}}|f(x)-g(x)|$.
Assuming that $X^{(n)}:=(X_1,\,\ldots,\,X_n)$ are i.i.d. observations from an unknown density $f_0$,
which may or may not be itself a kernel mixture, we
analyze contraction properties of the posterior distribution
\[(\Pi\times G)(B|X^{(n)})\propto \int_B\prod_{i=1}^nf_{F,\,\sigma}(X_i)\,\mathrm{d}(\Pi\times G)(F,\,\sigma),\qquad\mbox{for any Borel set $B$},\]
under regularity conditions on the prior $\Pi\times G$ and the sampling density $f_0$.
A sequence of positive numbers $\varepsilon_{n,p}\rightarrow0$ and such that $n\varepsilon_{n,p}^2\rightarrow\infty$, as $n\rightarrow\infty$,
is an \emph{upper bound} on the posterior rate of contraction
relative to the $L^p$-metric, $p\in[1,\,\infty]$, if, for a finite constant $M>0$,
$(\Pi\times G)((F,\,\sigma):\,\|f_{F,\,\sigma}-f_0\|_p\geq M\varepsilon_{n,p}|X^{(n)})\rightarrow0$ in $P_0^n$-probability, where $P_0^n$ stands for the joint law of the first $n$
coordinate projections of the infinite product probability measure $P_0^\infty$.
In the following section, we present general results on posterior contraction rates
for kernel mixture priors.

\section{Posterior contraction rates for kernel mixtures}~\label{sec:second}
In this section, we present a theorem providing sufficient conditions
for assessing posterior contraction rates in $L^p$-metrics,
$p\in[2,\,\infty]$, for super-smooth kernel mixture priors.
Results for specific priors on the mixing distribution belonging
to the class of species sampling models, which are useful
in concrete applications, are later exposed in Section~\ref{sec:third}.
To describe regularity properties of the sampling density,
we use a general approximation scheme
in function spaces, based on integrating a kernel-type function $K_j(x,\,y)$
against a density $f$, that is, $K_j(f):=\int K_j(\cdot,\,y)\,\mathrm{d}y$.
The \emph{sinc} kernel
\[\operatorname{sinc}(x):=
\left\{
  \begin{array}{cl}
  (\sin x)/(\pi x), & \hbox{\mbox{if } $x\neq0$,}\\
  1/\pi, & \hbox{\mbox{if } $x=0$,}
  \end{array}
\right.
\]
turns out to play a key role in characterizing regular
densities in terms of their approximation properties.
This is an unconventional kernel, \emph{i.e.}, it may take negative values, it is Riemann integrable with
$\int\operatorname{sinc}\,\mathrm{d}\lambda=1$, but not Lebesgue integrable,
$\operatorname{sinc}\notin L^1(\mathbb{R})$, it has Fourier transform identically equal to $1$
on $[-1,\,1]$ and vanishing outside it. The key role of the
$\operatorname{sinc}$ kernel in density estimation is known since the work of
Davis~\cite{Davis}, who showed that, for the $\operatorname{sinc}$ kernel density estimator,
the optimal MISE is of order $O(n^{-1}(\log n)^{1/r})$ for estimands
satisfying \eqref{expdecr} with characteristic exponent $r$.

Regularity of the overall prior is expressed through the usual
small ball probability condition involving Kullback-Leibler
type neighborhoods of $f_0$, \emph{i.e.}, $B_{\operatorname{KL}}(f_0;\,\varepsilon):=\{(F,\,\sigma):\,\operatorname{KL}(f_0;\,f_{F,\,\sigma})\leq \varepsilon,\,\operatorname{E}_0[(\log(f_{F,\,\sigma}/f_0))^2]\leq \varepsilon\}$, where
$\operatorname{KL}(\cdot;\,\cdot)$ denotes the Kullback-Leibler divergence, as well as
through the following assumption on $G$.
\begin{itemize}
\item[$\quad(\mathrm{A_0})$] The prior distribution $G$ for $\sigma$ has
a continuous and positive Lebesgue density $g$ on $(0,\,\infty)$ such that, for constants
$C_1,\,C_2,\,D_1,\,D_2>0$, $s,\,t\geq0$ and $0<\gamma\leq\infty$,
\[C_1\sigma^{-s}\exp{(-D_1\sigma^{-\gamma}(\log(1/\sigma))^t)}\leq g(\sigma)\leq C_2\sigma^{-s}
\exp{(-D_2\sigma^{-\gamma}(\log(1/\sigma))^t)}\]
for all $\sigma$ in a neighborhood of $0$.
\end{itemize}
An inverse-gamma distribution $\mathrm{IG}(\nu,\,\lambda)$
is an eligible prior on $\sigma$ satisfying assumption
$(\mathrm{A_0})$ for $s=\nu+1$, $t=0$ and $\gamma=1$.

\begin{thm}\label{ThsuperKs}
Let $K\in\mathcal{A}^{\rho,\,r,\,L}(\mathbb{R})$
for some constants $\rho,\,r,\,L>0$. Let $\tilde{\varepsilon}_n$ be a sequence such that
$\tilde{\varepsilon}_n\rightarrow0$ and
$n\tilde{\varepsilon}^2_n\rightarrow\infty$ as $n\rightarrow\infty$.
For each $p\in[2,\,\infty]$, let $\varepsilon_{n,p}:=\tilde{\varepsilon}_n(n\tilde{\varepsilon}_n^2)^{(1-1/p)/2}$.
Suppose that $f_0\in L^p(\mathbb{R})$ with
$\|f_0\ast \operatorname{sinc}_{2^{-J_n}}-f_0\|_p=O(\varepsilon_{n,p})$, for $2^{J_n}=O(n\tilde{\varepsilon}_n^2)$,
is such that
\begin{equation}\label{KLne}
(\Pi\times G)(B_{\mathrm{KL}}(f_0;\,\tilde{\varepsilon}_n^2))\gtrsim \exp{(-Cn\tilde{\varepsilon}_n^2)}
\qquad\mbox{for some constant $C>0$},
\end{equation}
where $G$ satisfies assumption $(\mathrm{A_0})$ with $s\geq0$, $t\geq r^{-1}$ if $\gamma=1$,
$t=0$ if $\gamma\in(1,\,\infty]$ such that $n\tilde{\varepsilon}_n^2\gtrsim (\log n)^{1/[r(1-1/\gamma)]}$.
Then, there exists a finite constant $M>0$ such that
\[(\Pi\times G)((F,\,\sigma):\,\|f_{F,\,\sigma}-f_0\|_p\geq M\varepsilon_{n,p}|X^{(n)})\rightarrow0\qquad\mbox{ in $P_0^n$-probability}.\]
\end{thm}
The assertion is an in-probability statement that the posterior
mass outside an $L^p$-norm ball of radius a large multiple $M$ of $\varepsilon_n$ is approximately
zero. Assumption \eqref{KLne}, which is the usual small ball
probability condition, as discussed in Ghosal \emph{et al.}~\cite{GGvdV00}, page~504,
is the essential one: the prior concentration rate is the only determinant of the posterior convergence rate
at regular densities having approximation error of the same order
against the sinc kernel-type approximant. Densities in $\mathcal{A}^{\rho,\,r,\,L}(\mathbb{R})$
meet this requirement. For concreteness, the regularity condition on $f_0$ has been
stated in terms of the $\operatorname{sinc}$ kernel, but
any continuous super-kernel $S$, with bounded $p$-variation for some finite $p\geq1$,
such that $S\in L^\infty(\mathbb{R})\cap L^2(\mathbb{R})$ can be employed,
cf. Subsection~\ref{appendixA}.

The theorem yields optimal (up to a $\log n$-term) rates
when the prior concentration rate is nearly parametric.
When $f_0$ is ordinary smooth, even if the prior concentration rate is minimax-optimal
(up to a logarithmic factor), sub-optimal posterior contraction rates are obtained.
Nonetheless, the result has an intrinsic value.
When the employed kernel has Fourier transform
decreasing at an exponential power rate and $f_0$ is itself a kernel mixture with compactly supported mixing distribution, Theorem~\ref{ThsuperKs} yields rates of contraction in the Wasserstein metric of order $2$
for the posterior on the mixing. We hereafter introduce the Wasserstein distance.
Let $(\Theta,\,d)$, $\Theta\subseteq\mathbb{R}$,
be a measurable metric space with the Borel $\sigma$-field. For $p\geq1$,
define the \emph{Wasserstein distance of order $p$}
between any two Borel probability measures $\mu$ and $\nu$ on $\Theta$ with finite $p$th-moment
(\emph{i.e.}, $\int_\Theta d^p(x,\,x_0)\,\mathrm{d}\mu(x)<\infty$ for some (and hence any) $x_0$ in $\Theta$)
as $W_p(\mu,\,\nu):=(\inf_{\gamma\in\Gamma(\mu,\,\nu)}\int_{\Theta \times \Theta} d^p(x,\,y)\,\mathrm{d}\gamma (x,\,y) )^{1/p}$, where $\gamma$ runs over the set $\Gamma(\mu,\,\nu)$ of all joint probability measures on $\Theta \times \Theta$
with marginal distributions $\mu$ and $\nu$. When $p=2$, we take $d$ to be the Euclidean distance on $\Theta$.
From the definition, $W_p(\mu,\,\nu)\in[0,\,\textrm{diam}(\Theta)]$, where $\textrm{diam}(\Theta)$ denotes
the diameter of $\Theta$. If $\Theta$ is compact, then $\textrm{diam}(\Theta)<\infty$.

\begin{cor}\label{Mixing}
Let $K$ be a symmetric density around $0$ such that
\begin{equation}\label{LBK}
\mbox{for some constants $\rho,\,r>0$,}\qquad |\hat{K}(t)|\asymp e^{-(\rho t)^r}
\qquad\mbox{as $|t|\rightarrow\infty$.}
\end{equation}
Suppose that $f_0=f_{F_0,\,1}=F_0\ast K_1$, with $F_0$ supported on some compact set
$\Theta\subset\mathbb{R}$. Let $\Pi$ be a prior on $\mathscr{M}(\Theta)$.
If condition \eqref{KLne} is satisfied for a sequence $\tilde{\varepsilon}_n$ such that $n\tilde{\varepsilon}_n^2\gtrsim (\log n)^{1/r}$,
then, for a sufficiently large constant $M'>0$,
\[\Pi(F:\,W_2(F,\,F_0)\geq M'(\log n)^{-1/r}|X^{(n)})\rightarrow0\qquad\mbox{ in $P_0^n$-probability}.\]
\end{cor}
In virtue of Theorem~\ref{ThsuperKs}, condition \eqref{KLne}, combined with \eqref{LBK}, implies
that the posterior for the mixture density concentrates
on a sup-norm ball centered at $f_0$, which is \emph{in} the model,
with probability approaching $1$. This assertion translates into a parallel
statement on the rate of contraction, relative to the Wasserstein metric of order $2$,
for the posterior on the mixing distribution.
The resulting rate only depends on the characteristic exponent $r$ of the Fourier transform of the kernel, so that
the greater $r$, the smoother the kernel, the more difficult to
recover the mixing distribution and the slower the rate.
The open question remains whether this rate is optimal.
Posterior contraction rates for the mixing distribution in Wasserstein metrics
have been recently investigated by Nguyen~\cite{Ng?}, who insightfully argues how
convergence in Wasserstein metrics for discrete mixing measures
has a natural interpretation in terms of convergence of the single atoms
providing support for the measures. He states sufficient entropy and remaining mass conditions in the
spirit of Ghosal \emph{et al.}~\cite{GGvdV00}, but in terms of the Wasserstein distance on mixing distributions
as opposed to the Hellinger or $L^1$-distance
on mixture densities. The result of Corollary~\ref{Mixing} allows to derive the posterior
contraction rate in the Wasserstein metric of order $2$ only from
the prior concentration rate
and is more general than Theorem~6 in the above mentioned paper, whose scope is confined to
Dirichlet process kernel mixtures.


\section{Posterior rates for specific priors on the mixing}\label{sec:third}
In this section, we derive posterior contraction rates for specific priors on the mixing distribution, \emph{i.e.},
the Pitman-Yor process, which renders the Dirichlet process as a special case,
and the normalized inverse-Gaussian process. These are popular process priors
and the techniques herein developed can be extended to other processes with similar features.


\subsection{Estimation of densities with a kernel mixture representation}
We begin the analysis from the simplest case where $f_0$ is itself a kernel mixture,
$f_0=f_{F_0,\,\sigma_0}$, with $F_0$ and $\sigma_0$
denoting the true values of the mixing distribution
and the scale, respectively. Considering this case helps developing techniques
that can be used for the case where $f_0$ is not necessarily a kernel mixture.
Results are obtained under the following assumptions.
\subsection*{Assumptions}
\begin{itemize}
\item[$\quad(\mathrm{A_1})$]
The kernel density $K:\,\mathbb{R}\rightarrow\mathbb{R}^+$
is symmetric around $0$, monotone decreasing in $|x|$ and satisfies the tail condition $
K(x)\gtrsim e^{-c|x|^\kappa}$ for large $|x|$, for some constants $c>0$ and $\kappa\in(0,\,\infty)$.
\item[$\quad(\mathrm{A_2})$]
The true mixing distribution $F_0$ satisfies the tail condition
\begin{equation}\label{tail1}
F_0(\theta:\,|\theta|>t)\lesssim
e^{-c_0t^{\varpi}} \qquad\textrm{for
large }\, t>0,
\end{equation}
for some constants $c_0>0$ and $\varpi\in(0,\,\infty]$.
\item[$\quad(\mathrm{A_3})$]
The base measure $\alpha$ has a continuous and positive Lebesgue density $\alpha'$ such that, for some
constants $b>0$ and $\delta\in(0,\,\infty]$, satisfies
\begin{equation}\label{eqX1}
\alpha'(\theta)\propto e^{-b|\theta|^\delta}\qquad\textrm{for large } |\theta|.
\end{equation}
\end{itemize}

Assumption $(\mathrm{A_1})$ prevents the use of oscillating
kernels. Assumptions $(\mathrm{A_2})$ and $(\mathrm{A_3})$
postulate standard requirements on the true mixing distribution
and the base measure density, respectively.

\subsubsection*{Stick-breaking processes and the Pitman-Yor process}
Stick-breaking processes form a popular class of priors, which
includes, as relevant special cases, the Dirichlet process,
the Pitman-Yor process, see Pitman and Yor~\cite{PY97}, the beta two-parameter process, see
Ishwaran and Zarepour~\cite{IZ2000}, Ishwaran and James~\cite{IJ2001}.
The trajectories of a stick-breaking process $F$ can be almost surely
represented as $F=\sum_{j=1}^\infty W_j\delta_{Z_j}$,
where $\delta_{Z_j}$ denotes a point mass at $Z_j$. The random variables
$(Z_j)_{j\geq1}$ are i.i.d. $\bar{\alpha}$, where $\bar{\alpha}$
is a non-atomic (\emph{i.e.}, $\bar{\alpha}(\{z\})=0$ for every $z\in\mathbb{R}$)
probability measure over $(\mathbb{R},\,\mathcal{B}(\mathbb{R}))$ defined as
$\bar{\alpha}:=\alpha/\alpha(\mathbb{R})$, $\alpha$ being a positive and finite measure.
The random variables $(W_j)_{j\geq1}$ are independent of $(Z_j)_{j\geq1}$ and such that
$W_j\in[0,\,1]$, with $\sum_{j=1}^\infty W_j=1$ almost surely. Furthermore,
\begin{equation}\label{stick}
W_1=V_1,\qquad W_j=V_j\prod_{h=1}^{j-1}(1-V_h),\qquad j\geq2,
\end{equation}
with $V_j|H_j\overset{\textrm{indep}}{\sim} H_j$, where $H_j$ is a probability measure on $[0,\,1]$.
A necessary and sufficient condition for $\sum_{j=1}^\infty W_j=1$ almost surely
is that $\sum_{j=1}^\infty \log(1-\operatorname{E}_{H_j}[V_j])=-\infty$,
see, \emph{e.g.}, Lemma~1 in Ishwaran and James~\cite{IJ2001}, pages~162 and 170.

A stick-breaking process where, for $d\in[0,\,1)$ and $c>-d$, $V_j\overset{\textrm{indep}}{\sim}\textrm{Beta}(1-d,\,c+dj)$, $j\in\mathbb{N}$,
is called the \emph{Pitman-Yor process} or the \emph{two-parameter Poisson-Dirichlet process}, denoted $F\sim\textrm{PY}(c,\,d,\,\bar{\alpha})$, with \emph{concentration} parameter $c$,
\emph{discount} parameter $d$ and \emph{base distribution} $\bar{\alpha}$:
\begin{eqnarray*}
F&\sim&\sum_{j=1}^\infty \pq{V_j\prod_{h=1}^{j-1}(1-V_h)}\delta_{Z_j}\\[-0.15cm]
V_j&\overset{\textrm{indep}}{\sim}&\textrm{Beta}(1-d,\,c+dj)\\[-0.15cm]
Z_j&\overset{\textrm{iid}}{\sim}&\bar{\alpha}.
\end{eqnarray*}
The case where $d=0$ and $c=\alpha(\mathbb{R})$ returns the Dirichlet process with base measure $\alpha$.
In the Pitman-Yor process, the weights
$(V_j\prod_{h=1}^{j-1}(1-V_h))_{j\geq1}$ are the weights of the process in size-biased order.
When $c=0$, the Pitman-Yor process reduces to a stable process.
When $c=0$ and $d=1/2$, the stable process is a normalized
inverse-gamma process. There are no known analytic expressions for its finite-dimensional
distributions, except when $d=0$ or $d=1/2$.

The Dirichlet process, the Pitman-Yor process with $d=1/2$ and the normalized inverse-Gaussian
process are the only known processes for which explicit expressions of the finite-dimensional
distributions are available.
\subsubsection*{Normalized inverse-Gaussian process}\label{sec:N-IG}
Considered a space $\mathbb{X}$ with a $\sigma$-algebra $\mathcal{A}$ of subsets of $\mathbb{X}$,
let $\alpha$ be a finite and positive measure on $(\mathbb{X},\,\mathcal{A})$. Following
Lijoi \emph{et al.}~\cite{LMP05}, a random probability measure $F$ is called a \emph{normalized inverse-Gaussian (N-IG) process} on $(\mathbb{X},\,\mathcal{A})$, with parameter $\alpha$, denoted $\operatorname{N-IG}(\alpha)$, if, for every finite measurable partition $A_1,\,\ldots,\,A_N$ of $\mathbb{X}$, the vector $(F(A_1),\,\ldots,\,F(A_N))$ has a N-IG distribution with parameters $(\alpha(A_1),\,\ldots,\, \alpha(A_N))$, cf. \eqref{factors}.

\medskip

The following theorem extends results of Ghosal and van
der Vaart~\cite{GvdV01} on posterior contraction rates for Dirichlet process Gaussian mixtures to
Pitman-Yor kernel mixtures in $L^p$-metrics, $p\in[1,\,\infty]$.


For given reals $\kappa,\,r>0$, let $\varpi$ be such that
\begin{equation}\label{varpiexpr}
\max\{\kappa,\,[1+1_{(1,\,\infty)}(r)/(r-1)]\}\leq \varpi\leq \infty
\end{equation}
and let $\tau$ be defined as
\begin{equation}\label{tauexpr}
\tau:=1+\pq{1/r-\pt{1-1_{(0,\,\infty)}(\varpi)/\varpi}}1_{(0,\,1]}(r)/2.
\end{equation}
Condition \eqref{varpiexpr} requires a
matching between the tail decay speed of the kernel $K$
and that of the true mixing distribution $F_0$.

\begin{thm}\label{superKs}
Let $K\in\mathcal{A}^{\rho,\,r,\,L}(\mathbb{R})$,
for some constants $\rho,\,r,\,L>0$, be as in assumption $(\mathrm{A_1})$.
Suppose that $f_0=f_{F_0,\,\sigma_0}=F_0*K_{\sigma_0}$,
with
\begin{itemize}
\item[$(i)$] $F_0$ satisfying assumption $(\mathrm{A_2})$
for some constants $c_0>0$ and $\varpi$ as in \eqref{varpiexpr}.
\end{itemize}
Let $F\sim\mathrm{PY}(c,\,d,\,\bar{\alpha})$, with $d\in[0,\,1)$ and $c>-d$. Alternatively,
let $F\sim\operatorname{N-IG}(\alpha)$. Assume that
\begin{itemize}
\item[$(ii)$] $\alpha$ satisfies assumption $(\mathrm{A_3})$ for some constants
$b>0$ and $\delta\in(0,\,\infty)$, with $\delta\leq\varpi$ when $\varpi<\infty$;
\item[$(iii)$] $G$ satisfies assumption $(\mathrm{A_0})$, with $s\geq0$, $t\geq0$ if $p=1$,
\begin{equation*}
\left\{
    \begin{array}{ll}
      t\geq r^{-1}, & \hbox{ for\,\, $\,\,\,\,\,\,\,\,\,\,\,\gamma=1$,}\\[1pt]
      t=0, & \hbox{ for\,\, $1<\gamma\leq\infty$\,\, and \,\, $\gamma\geq\{1-\{2r[\tau+(\tau-1/2)1_{(0,\,\infty)}(d)]\}^{-1}\}^{-1}$,}
    \end{array}
  \right.
\end{equation*}
where $\tau$ is as in \eqref{tauexpr}, if $p\in[2,\,\infty]$.
Furthermore, if $p=1$, for some constant $\varrho\in(0,\,\infty]$, $1-G(\sigma)\lesssim
\sigma^{-\varrho}$ as $\sigma\rightarrow\infty$.
\end{itemize}
Then, for $p=1$ or $p\in[2,\,\infty]$, the posterior rate of convergence
$\varepsilon_{n,p}$ relative to the $L^p$-metric
is $n^{-1/2}(\log n)^{\mu}$ with a suitable constant
$\mu>0$ possibly depending on $p$. If conditions specific of the cases
$p=1$ and $p=2$ are simultaneously met, then, for every $p\in(0,\,1)$,
$\varepsilon_{n,p}\leq(\varepsilon_{n,1}\vee \varepsilon_{n,2})$.
\end{thm}

Theorem~\ref{superKs}, whose proof is postponed to Subsection~\ref{ana-loc},
shows that a nearly parametric rate is achievable, irrespective of the
tail behavior of the kernel (hence of the sampling density $f_0$), heavy-tailed distributions, like
Student's-$t$, which play a crucial role in modeling certain phenomena,
being admitted. Estimation of heavy-tailed distributions is not covered by Theorem~\ref{analytic} on adaptation,
which, by requiring $f_0$ to have sub-exponential tails, rules out these distributions.

\subsection{Adaptive estimation of analytic densities}
In this section, we study adaptive estimation of analytic densities using
Gaussian mixtures.
We assume that $f_0$ satisfies the following conditions,
where $C^{\omega}(\mathbb{R})$ denotes the class of analytic functions on $\mathbb{R}$.
\begin{itemize}
\item[$(a)$]\emph{Smoothness}: $f_0\in C^{\omega}(\mathbb{R})\cap
\mathcal{A}^{\rho_0,\,r_0,\,L_0}(\mathbb{R})$ for some constants $\rho_0>0$, $r_0\geq1$ and $L_0>0$.
Furthermore, $\sum_{j=1}^\infty\operatorname{E}_0[|f_0^{(j)}(X)/(C_{0j}f_0(X))|^{r_0/j}]<\infty$,
where $(C_{0j}^{r_0/j})_{j\geq1}$ is a sequence of positive reals bounded below away from zero and above from infinity.
\item[$(b)$]\emph{Monotonicity}: $f_0$ is a strictly positive and bounded density,
non-decreasing on $(-\infty,\,a)$, non-increasing on $(b,\,\infty)$
and such that $f_0\geq \ell_0>0$ on $[a,\,b]$.
\item[$(c)$] \emph{Tails}: there exist finite constants
$M_0,\,c_0,\,\varpi>0$ such that $f_0(x)\leq M_0 e^{-c_0|x|^\varpi}$ for large $|x|$.
\end{itemize}


To prove that contraction rates of posterior distributions corresponding to a Pitman-Yor or a N-IG process mixture of Gaussians adapt to the \vir{analytic smoothness} of $f_0$, the key step is the approximation of $f_0$ by a continuous mixture, which is then discretized to have a sufficiently restricted number of support points, see Lemma~\ref{appro}. We suspect that this step is only possible under assumption $(c)$ that $f_0$ has sub-exponential tails: this condition seems to be necessary to obtain a nearly parametric rate because, when restricting to a symmetric compact set, it allows to take the endpoint of the order $O(\log(1/\varepsilon))$, thus finding a finite mixture with a small number of points. A density with polynomially decreasing tails would incur an additional factor of $\varepsilon^{-k}$ and a genuine power of $n$ would be lost in the prior as well as in the posterior concentration rate.
The key step is the construction of a (not necessarily non-negative) function that uniformly approximates $f_0$, see Lemma~\ref{approxlem}. By suitably modifying this function, we obtain a density with the same
approximation error in Kullback-Leibler divergence, which is needed for the prior concentration rate.
The general strategy is similar to that adopted by Kruijer \emph{et al.}~\cite{KRvdV10},
but the iterative procedure they use to construct the approximant turns out to be inefficient
because of the infinite degree of smoothness of $f_0$. As far as we are aware, the approximation result of Lemma~\ref{approxlem}, involving the use of the sinc kernel, is novel.
Once a finite mixture is derived, we need to show that there exists a whole set of finite mixtures,
close to it and contained in a Kullback-Leibler type ball, receiving enough prior mass.
We are now in a position to state the result.


\begin{thm}\label{analytic}
Suppose that $f_0$ satisfies conditions $(a)$-$(c)$.
Let the model be $f_{F,\,\sigma}=F\ast \phi_\sigma$, with
$F\sim\operatorname{PY}(c,\,d,\,\bar{\alpha})$, for $d\in[0,\,1)$ and $c>-d$. Alternatively,
let $F\sim\operatorname{N-IG}(\alpha)$. Assume that
\begin{itemize}
\item[$(i)$] $\alpha$ satisfies $(\mathrm{A_3})$ for some constants
$b>0$ and $\delta\in(0,\,2]$;
\item[$(ii)$] $G$ satisfies condition $(\mathrm{A_0})$, with
$s\geq0$, $\gamma=1$, $t\geq0$ if $p=1$, $t\geq\frac{1}{2}$ if $p\in[2,\,\infty]$.
Furthermore, if $p=1$, for some constant $\varrho\in(0,\,\infty]$, $1-G(\sigma)\lesssim
\sigma^{-\varrho}$ as $\sigma\rightarrow\infty$.
\end{itemize}
Then, the posterior rate of convergence relative to the $L^p$-metric, denoted $\varepsilon_{n,p}$, is
\[\varepsilon_{n,p}=
\left\{
\begin{array}{ll}
n^{-1/2}(\log n)^{\frac{1}{2}+\{\frac{1}{2}\vee[2(1+\frac{1}{\delta})\psi(r_0,\,d)]\}}, & \hbox{for \, $p=1$,} \\[2pt]
n^{-1/2}(\log n)^{(2-1/p)\psi(r_0,\,d)}, & \hbox{for \, $p\in[2,\,\infty]$,}
\end{array}
\right.
\]
where
\begin{equation}\label{power}
\psi(r_0,\,d):=1/2+\{(t/2)\vee[((\varpi\wedge2 )^{-1}+(r_0\wedge 2)^{-1})(1+1_{(0,\,\infty)(d)})]\}.
\end{equation}
If conditions specific of the cases
$p=1$ and $p=2$ are simultaneously satisfied, then, for every $p\in(0,\,1)$,
$\varepsilon_{n,p}\leq(\varepsilon_{n,1}\vee \varepsilon_{n,2})$.
\end{thm}
Given that the power of $n$ is fixed at $-\frac{1}{2}$, the most important factor in the rate is the
logarithmic power which adapts to the characteristic exponent $r_0$ of $f_0$.
A main implication of Theorem~\ref{analytic},
whose proof is postponed to Section~\ref{sec:ana},
is that the choice of the kernel is not an issue in Bayesian density estimation.
A well-known problem with the use of Gaussian convolutions is that the approximation error of a smooth density can only be of the order $O(\sigma^2)$, even if the density has greater smoothness. The approximation can be improved using higher-order kernels, but the resulting convolution is not guaranteed to be everywhere non-negative which, in a frequentist approach, translates into a non-\emph{bona fide} estimator. This is not a problem in a Bayesian framework because to have adaptation it suffices that the prior support contains a set of densities close to $f_0$ receiving enough mass, which is the case when endowing the mixing distribution with a Pitman-Yor or a N-IG process prior.
\subsection{Adaptive estimation over Sobolev classes}\label{sec:fourth}
In this section, we study adaptive estimation of densities in Sobolev classes
using Gaussian mixtures. We assume that $f_0$ satisfies the following condition, where $W^{k_0,\,2}(\mathbb{R}):=\{f\in L^2(\mathbb{R}):\,\int(1+t^2)^{k_0}|\hat{f}(t)|^2\,\mathrm{d}t<\infty\}$ denotes the
Sobolev space of order $k_0\in\mathbb{N}$.
\begin{itemize}
\item[$(a')$]\emph{Smoothness}: $f_0\in W^{k_0,\,2}(\mathbb{R})$, $k_0\in\mathbb{N}$, with
$\operatorname{E}_0[|((f_0^{(j)}\ast S_\sigma)/f_0)(X)|^{(2k_0-1)/j}]$ for every $j=1,\,\ldots,\,k_0-1$,
where $S$ is any superkernel.
\end{itemize}

The following theorem, whose proof is deferred to Section~\ref{ada-sobo},
asserts that, whatever the \vir{Sobolev smoothness} $k_0$
of $f_0$, the posterior corresponding to a Dirichlet or a N-IG process
mixture of Gaussians contracts at a rate at least as fast
as $n^{-(1-1/2k_0)/2}(\log n)^\kappa$, with $\kappa>0$, in all $L^p$-norms for $p\in[1,\,2]$.
\begin{thm}\label{adaptiveSobolev}
Suppose that $f_0$ satisfies conditions $(a')$, $(b)$-$(c)$ and
the integrability condition in Lemma~\ref{appro2}.
Let the model be $f_{F,\,\sigma}=F\ast \phi_\sigma$, with
$F\sim\operatorname{DP}(\alpha)$ or $F\sim\operatorname{N-IG}(\alpha)$. Assume that
\begin{itemize}
\item[$(i)$] $\alpha$ satisfies $(\mathrm{A_3})$ for some constants
$b>0$ and $\delta\in(0,\,2]$;
\item[$(ii)$] $G$ satisfies condition $(\mathrm{A_0})$, with
$s\geq0$, $\gamma=1$, $t\geq0$ if $p=1$, $t\geq \frac{1}{2}$ if $p=2$.
Furthermore, if $p=1$, $G$ is supported on $(0,\,S]$, with $S\geq1$.
\end{itemize}
Then, the posterior rate of convergence relative to the $L^p$-metric, denoted $\varepsilon_{n,p}$, is
\begin{equation}\label{Ka}
\varepsilon_{n,p}=
\left\{
\begin{array}{ll}
n^{-(1-1/2k_0)/2}(\log n)^{\tau+5/4}, & \hbox{for \, $p=1$,} \\[2pt]
n^{-(1-1/2k_0)/2}(\log n)^{\tau}, & \hbox{for \, $p=2$,}
\end{array}
\right. \mbox{ where } \tau:=5(1-1/2k_0)/4.
\end{equation}
If conditions specific of the cases
$p=1$ and $p=2$ are simultaneously satisfied, then, for every $p\in(0,\,1)$,
$\varepsilon_{n,p}\leq n^{-(1-1/2k_0)/2}(\log n)^{\tau+5/4}$.
\end{thm}

A few comments are in order here. Slower rates are found when endowing
the mixing distribution with a Pitman-Yor process having strictly
positive discount parameter $d$ because small balls do not receive
enough prior mass. The open question is whether
posterior contraction rates under a Pitman-Yor process prior
are indeed sub-optimal. Furthermore, rates in $L^p$-norms
deteriorate by a genuine power of $n$ for $p>2$.


\section{Prior estimates}\label{sec:estimates}
Estimates, under different priors, of the probability of an $L^1$-ball
are essential to evaluate the prior mass of Kullback-Leibler type balls as in \eqref{KLne}.
While for the N-IG process, the expression of the finite-dimensional distributions can be used
as in Lemma~A.1 of Ghosal \emph{et al.}~\cite{GGvdV00}, pages 518--519,
which deals with the Dirichlet process, for the Pitman-Yor process,
the stick-breaking representation can be exploited
to obtain separate (lower) bounds on the probabilities of $L^1$-balls of the
mixing weights and the locations.
\subsection{Pitman-Yor process}
\begin{lem}\label{PDP1}
Let $F\sim\mathrm{PY}(c,\,d,\,\bar{\alpha})$, with $d\in[0,\,1)$ and $c>-d$.
Let $F'=\sum_{j=1}^Np_j\delta_{z_j}$, $1\leq N<\infty$,
be a probability measure on $\mathbb{R}$ with $p_1\geq p_2\geq\,\ldots\,\geq p_N>0$.
Define $v_1:=p_1$ and $v_j:=p_j[\prod_{h=1}^{j-1}(1-v_h)]^{-1}$, $2\leq j\leq N$.
Let $v_{\max}:=\max_{1\leq j\leq N}v_j$. For $\varepsilon\in(0,\,1)$, let
$U:=(\sum_{j=1}^N\sum_{h=1}^j|V_h-v_h|\leq2\varepsilon,\,\,\,\min_{1\leq j\leq N}V_j>\varepsilon/N^2)$,
where the random variables $V_1,\,\ldots,\,V_N$ are those arising from the stick-breaking representation \eqref{stick}.
Then, there exist constants $c_1,\,C>0$ (depending only on $c$ and $d$) such that, for $(2\varepsilon/N^2)<(1-v_{\max})/2$,
$\operatorname{P}(U)\geq C\exp{(-c_1N\max\{\log(N/\varepsilon),\,dN\log(1/(1-v_{\max}))\})}$.
\end{lem}
\begin{proof}
If $|V_j-v_j|\leq2\varepsilon/N^2$ for $j=1,\,\ldots,\,N$, then
$\sum_{j=1}^N\sum_{h=1}^j|V_h-v_h|\leq2\varepsilon$. Thus,
$U$ is implied by
$V:=(|V_j-v_j|\leq2\varepsilon/N^2,\,\,\,V_j>\varepsilon/N^2,\,\,\, j=1,\,\ldots,\,N)$.
Let $l_j:=((v_j-2\varepsilon/N^2)\vee(\varepsilon/N^2))$ and $u_j:=((v_j+2\varepsilon/N^2)\wedge1)$,
$j=1,\,\ldots,\,N$. By assumption,
$V_j\overset{\textrm{indep}}{\sim}\textrm{Beta}(1-d,\,c+dj)$, $j\in\mathbb{N}$, thus, by
the identity $\Gamma(z+1)=z\Gamma(z)$, $z>0$,
\[\operatorname{P}(V)\geq
\frac{[\Gamma(1-d)]^{-N}\Gamma(c)c^N}
{\Gamma(c+dN)}\prod_{j=1}^N \int_{l_j}^{u_j}(1-v)^{c+dj-1}\,\mathrm{d}v.
\]
If $N\rightarrow\infty$ as $\varepsilon\rightarrow0$, using
$\Gamma(c+dN)\sim (2\pi)^{1/2}e^{-dN}(dN)^{dN+c-1/2}$,
\[\begin{split}
\operatorname{P}(V)
&\gtrsim\frac{[\Gamma(1-d)]^{-N}\Gamma(c)c^N(\varepsilon/N^2)^N}
{\Gamma(c+dN)}[1-((v_{\max}+2\varepsilon/N^2)\wedge1)]^{(c-1)N+dN(N+1)/2}\\
&\gtrsim\exp{(-c_1N\max\{\log(N/\varepsilon),\,dN\log(1/(1-v_{\max}))\})},
\end{split}
\]
provided $(2\varepsilon/N^2)<(1-v_{\max})/2$, where $v_{\max}\in(0,\,1)$
because of the positivity constraint on the mixing weights. Conclude by noting
that $\operatorname{P}(U)\geq\operatorname{P}(V)$.
\end{proof}
\begin{rmk}
\emph{
For $d=0$, if $N=O((1/\varepsilon)^\xi)$ for some $\xi>0$, we have
$\operatorname{P}(U)\gtrsim \exp{(-c_1N\log(1/\varepsilon))}$,
which agrees with the estimate known for a Dirichlet process,
cf. Lemma~6.1 in Ghosal \emph{et al}.~\cite{GGvdV00}, pages~518--519,
or Lemma~A.1 in Ghosal~\cite{G01}, pages~1278--1279.}
\end{rmk}
\begin{lem}\label{PDP2}
Let $F\sim\mathrm{PY}(c,\,d,\,\bar{\alpha})$, with $d\in[0,\,1)$, $c>-d$
and the (un-normalized) base measure $\alpha=\alpha(\mathbb{R})\bar{\alpha}$
satisfying $(\mathrm{A_3})$ for constants $b>0$ and $\delta\in(0,\,\infty)$.
For $\varepsilon\in(0,\,1)$, let $F'=\sum_{j=1}^Np_j\delta_{z_j}$, $1\leq N<\infty$,
be a probability measure with
$\mathrm{supp}(F')\subseteq[-a,\,a]$ for $a$ large enough. Then,
$\operatorname{P}(\sum_{j=1}^N|Z_j-z_j|\leq\varepsilon)\gtrsim
\exp{(-N[\log(N\alpha(\mathbb{R})/(2\varepsilon))+ba^\delta])}$.
\end{lem}
\begin{proof}
If $|Z_j-z_j|\leq\varepsilon/N$ for every $j=1,\,\ldots,\,N$, then $\sum_{j=1}^N|Z_j-z_j|\leq\varepsilon$.
Since $Z_1,\,\ldots,\,Z_N\overset{\textrm{iid}}{\sim}\bar{\alpha}$, we have
$\operatorname{P}(\sum_{j=1}^N|Z_j-z_j|\leq\varepsilon)\geq\prod_{j=1}^N\int_{z_j-\varepsilon/N}^{z_j+\varepsilon/N}
[\alpha'(z)/\alpha(\mathbb{R})]\,\mathrm{d}z
\gtrsim\exp{(-N[\log(N\alpha(\mathbb{R})/(2\varepsilon))+ba^\delta])}$,
where the last inequality follows from $(\mathrm{A_3})$ and the assumption that
$a$ is large enough.
\end{proof}
\subsection{Normalized inverse-Gaussian process}
We preliminarily recall the definition of the N-IG distribution.
The random vector $(Z_1,\,\ldots,\,Z_N)$, $N\geq2$, has a N-IG
distribution with parameters $(\alpha_1,\,\ldots,\,\alpha_N)$, where $\alpha_j\geq0$ for every
$j=1,\,\ldots,\,N$ and $\alpha_j>0$ for at least one $j$, denoted $\textrm{N-IG}(\alpha_1,\,\ldots,\,\alpha_N)$,
if it has density over the unit $(N-1)$-simplex $\Delta^{N-1}$
\begin{eqnarray}\label{factors}
f(z_1,\,\ldots,\,z_{N-1})&=&
\frac{e^{\sum_{j=1}^N\alpha_j}\prod_{j=1}^N\alpha_j}{2^{N/2-1}\pi^{N/2}}\,\times\,
K_{-N/2}(\sqrt{\mathcal{A}_N(z_1,\,\ldots,\,z_{N-1})})\nonumber\\&&\qquad
\times\,
\pt{\mathcal{A}_N(z_1,\,\ldots,\,z_{N-1})}^{-N/4}\nonumber\\&&\qquad\times\,
[z_1\times\ldots\times z_{N-1}\times(1-z_1-\,\ldots\,-z_{N-1})]^{-3/2}\nonumber\\
&=:&\prod_{r=1}^4h_r(z_1,\,\ldots,\,z_{N-1}),
\end{eqnarray}
where $K_{-N/2}(\cdot)$ is the modified Bessel function of the second kind
and $\mathcal{A}_N(z_1,\,\ldots,\,z_{N-1}):=\sum_{j=1}^{N-1}(\alpha_j^2/z_j)+\alpha_N^2/(1-\sum_{j=1}^{N-1}z_j)$.
We prove an analogue of Lemma 6.1 in Ghosal \emph{et al.}~\cite{GGvdV00}, pages~518--519,
or Lemma~A.1 in Ghosal~\cite{G01}, pages~1278--1279, which
provides an estimate of the probability of an $L^1$-ball in $\mathbb{R}^N$
under the N-IG distribution.
\begin{lem}\label{N-IG}
Let $(Z_1,\,\ldots,\,Z_N)$ be distributed according
to the N-IG distribution with parameters $(\alpha_1,\,\ldots,\,\alpha_N)$.
Let $(z_{10},\,\ldots,\,z_{N0})\in\Delta^{N-1}$.
For $\varepsilon\in(0,\,1)$, let
$U:=(\sum_{j=1}^N|Z_j-z_{j0}|\leq2\varepsilon,\,\,\,\min_{1\leq j\leq N}Z_j>\varepsilon^2/2)$.
Assume that $A\varepsilon^b\leq\alpha_j\leq1$ for every $1\leq j\leq N$ and some constants
$A,\,b>0$. If $\min_{1\leq j\leq N}z_{j0}>\varepsilon$, there exist constants $c,\,C>0$ (depending only on
$A$, $b$ and $m:=\sum_{j=1}^N\alpha_j$) such that, for $\varepsilon\leq1/N$ and $N\rightarrow\infty$ as $\varepsilon\rightarrow0$, $\operatorname{P}(U)\geq C\exp{(-cN\max\{\log(1/\varepsilon),\,
\log(1/(\min_{1\leq j\leq N}z_{j0}-\varepsilon))\})}$.
\end{lem}
\begin{proof}
As in the proof of Lemma~6.1 in Ghosal \emph{et al.}~\cite{GGvdV00}, pages~518--519,
we can assume that $z_{N0}\geq 1/N$.
If $|Z_j-z_{j0}|\leq \varepsilon^2$ for every $j=1,\,\ldots,\,N-1$,
then $\sum_{j=1}^N|Z_j-z_{j0}|\leq 2\varepsilon$
and $Z_N\geq \varepsilon^2>\varepsilon^2/2$. Therefore, $U$ is implied by
$V:=(|Z_j-z_{j0}|\leq\varepsilon^2,\,\,\,Z_j>\varepsilon^2/2,\,\,\, j=1,\,\ldots,\,N-1)$.
For $l_j:=((z_{j0}-\varepsilon^2)\vee(\varepsilon^2/2))$ and $u_j:=((z_{j0}+\varepsilon^2)\wedge1)$,
$j=1,\,\ldots,\,N-1$, $\operatorname{P}(V)=\int_{l_1}^{u_1}
\cdots\int_{l_{N-1}}^{u_{N-1}} f(z_1,\,\ldots,\,z_{N-1})\,\mathrm{d}z_1\,\cdots\,\mathrm{d}z_{N-1}$,
where $f=\prod_{r=1}^4h_r$, with the $h_r$'s as in \eqref{factors}. Then,
\[\begin{split}
\operatorname{P}(V)&\geq
\frac{e^m(A\varepsilon^b)^N}{2^{N/2-1}\pi^{N/2}}\,\times\,(em)^{-N/2}
\pt{\min_{1\leq j\leq N}z_{j0}-\varepsilon}^{N/2}\,\times\,\pt{\frac{\varepsilon^2}{2}}^{N-1}\\&\gtrsim \exp{\pt{-cN\max\pg{\log(1/\varepsilon),\,
\log\pt{1\large/\big(\min_{1\leq j\leq N}z_{j0}-\varepsilon\big)}}}},
\end{split}
\]
where $h_1$ is bounded below using the constraint $\alpha_j\geq A\varepsilon^b$, while $h_4\geq1$ because every $z_j\leq1$, $j=1,\,\ldots,\,N$.
To bound below $h_2$, first note that $K_{-N/2}(\cdot)=K_{N/2}(\cdot)$ (see 9.6.6 in Abramowitz and Stegun~\cite{AS}, page~375). Since, for $\varepsilon$ small enough,
\[(\mathcal{A}_N(z_1,\,\ldots,\,z_{N-1}))^{1/2}\leq m^{1/2}\pt{\min_{1\leq j\leq N}z_{j0}-\varepsilon}^{-1/2}\ll (N/2+1)^{1/2},\]
the approximation $h_2\sim 2^{N/2-1}\Gamma(N/2)(\mathcal{A}_N(z_1,\,\ldots,\,z_{N-1}))^{-N/4}$
holds (\emph{ibidem}, formula 9.6.9). By Stirling's formula,
$h_2\gtrsim e^{-N/2}m^{-N/4}(\min_{1\leq j\leq N}z_{j0}-\varepsilon)^{N/4}$. Consequently,
$h_2\times h_3\gtrsim (em)^{-N/2}(\min_{1\leq j\leq N}z_{j0}-\varepsilon)^{N/2}$.
\end{proof}
\section{Approximation results and proof of Theorem~\ref{analytic}}\label{sec:ana}
The main difficulty lies in finding a finite mixing
distribution with only $N(\varepsilon)\approx \log(1/\varepsilon)$ support points such
that the corresponding Gaussian mixture is within $\varepsilon$ Kullback-Leibler distance
from $f_0$. Such a finite mixing distribution may be found by matching
a certain number of moments of the \emph{ad hoc} constructed
mixing density with those of the finitely supported mixing distribution.
The crux is the approximation of an analytic density having exponentially decaying Fourier transform
by convolving the Gaussian kernel with an operator, whose expression resembles
a Taylor series expansion with suitably calibrated coefficients and
derivatives convolved with the $\operatorname{sinc}$ kernel.
Such a (not necessarily non-negative) function is
a linear combination, with coefficients summing up to $1$,
of iterated convolutions of $f_0$ with the Gaussian kernel.
Once this function is modified to be a density with the same tail behavior as $f_0$ and enjoying the same
approximation properties in the sup-norm and Kullback-Leibler divergence,
the re-normalized restriction to a compact set of the corresponding continuous mixture is discretized.

We begin by stating the result on the approximation
of analytic densities by convolutions with the Gaussian kernel.
Let $m_j:=\int y^j\phi(y)\,\mathrm{d}y$ denote the moment of order $j$ of a standard normal. For every $j\in\mathds{N}$, define two collections
of numbers $c_j$ and $d_j$. For $j=1$, set $c_1=d_1=0$. For $j=2$, set $c_2=0$
and $d_2=m_2/2!$. For every integer $j\geq3$,
\begin{equation}\label{d_j's}
c_j:=-\sum_{\substack{j=k+l\\k\geq1,\,l\geq1}}\frac{m_km_l}{k!l!},\qquad d_j:=\frac{(-1)^jm_j}{j!}+c_j.
\end{equation}
Note that the numbers $c_j$ and $d_j$ only depend on the moments of $\phi$.
Since moments of all odd orders are null for the Gaussian kernel, only numbers $d_{2j}$'s are non null.
For any real $\sigma>0$ and an infinitely differentiable function $f_0$, we define the transform
\[T_\sigma(f_0):=f_0-\sum_{j=1}^\infty d_j\sigma^j(f_0^{(j)}\ast\operatorname{sinc}_\sigma).\]
The following result holds.

\begin{lem}\label{approxlem}
Let $f_0\in C^{\omega}(\mathbb{R})\cap\mathcal{A}^{\rho_0,\,r_0,\,L_0}(\mathbb{R})$
for some constants $\rho_0,\,r_0,\,L_0>0$. For $\sigma>0$ small enough, whatever $\alpha\in(0,\,1)$,
\begin{equation}\label{sup-norm}
\| T_\sigma(f_0)\ast\phi_\sigma-f_0\|_\infty
\lesssim e^{-\alpha(\rho_0/\sigma)^{r_0}}1_{\{\infty\}}(S_{f_0})
\end{equation}
and
\begin{equation}\label{equality}
T_\sigma(f_0)=3f_0-3(f_0\ast\phi_\sigma)+f_0\ast\phi_\sigma\ast\phi_\sigma+O(e^{-\alpha(\rho_0/\sigma)^{r_0}}
1_{\{\infty\}}(S_{f_0})).
\end{equation}
Furthermore, $\int T_\sigma(f_0)\,\mathrm{d}\lambda=1+o(e^{-\alpha(\rho_0/\sigma)^{r_0}}1_{\{\infty\}}(S_{f_0}))$.
\end{lem}
\begin{proof}
By definition of $T_\sigma(f_0)$, Taylor's formula and the assumption that
$f_0\in C^{\omega}(\mathbb{R})$, for every $x\in\mathbb{R}$,
\[\begin{split}
&\hspace*{-0.5cm}(T_\sigma(f_0)\ast\phi_\sigma-f_0)(x)\\
&\hspace*{0.5cm}=\int\bigg[f_0(x-y)-f_0(x)-\sum_{j=1}^\infty d_j\sigma^j(f_0^{(j)}\ast \operatorname{sinc}_\sigma)(x-y)\bigg]\phi_\sigma(y)\,\mathrm{d}y\\
&\hspace*{0.5cm}=\sum_{j=1}^\infty\pt{\frac{(-1)^jm_j}{j!}\sigma^jf_0^{(j)}(x)-d_j\sigma^j(f_0^{(j)}\ast \operatorname{sinc}_\sigma\ast\phi_\sigma)(x)}\\
&\hspace*{0.5cm}=\sum_{j=1}^\infty\pt{\frac{(-1)^jm_j}{j!}\sigma^j(f_0^{(j)}-f_0^{(j)}\ast \operatorname{sinc}_\sigma\ast\phi_\sigma)(x)-c_j\sigma^j(f_0^{(j)}\ast\operatorname{sinc}_\sigma\ast\phi_\sigma)(x)},
\end{split}
\]
where, in the last line, the definition of the $d_j$'s has been used.
For every $j\in\mathbb{N}$,
\[
\begin{split}
&\hspace*{-0.7cm}(f_0^{(j)}-f_0^{(j)}\ast\operatorname{sinc}_\sigma\ast\phi_\sigma)(x)
\\ \quad&=\frac{1}{2\pi}\,
\int_{|t|>1/\sigma}(-it)^je^{-itx}\widehat{f_0}(t)\,\mathrm{d}t\\
&\hspace*{3cm} +\frac{1}{2\pi}
\int(-it)^je^{-itx}\widehat{f_0}(t)1_{[-1,\,1]}(\sigma t)\,\mathrm{d}t-
(f_0^{(j)}\ast\operatorname{sinc}_\sigma\ast\phi_\sigma)(x)\\
\quad&=: T_1(j,\,\sigma,\,x)+T_2(j,\,\sigma,\,x).
\end{split}
\]
By the Cauchy-Schwarz inequality and the assumption that $\widehat{f_0}$ satisfies \eqref{expdecr},
for $\sigma>0$ small enough, whatever $\alpha\in(0,\,1)$, we have
$T_1(j,\,\sigma,\,x)\lesssim\sigma^{-j}e^{-\alpha(\rho_0/\sigma)^{r_0}}1_{\{\infty\}}(S_{f_0})$.
Thus, $\sum_{j=1}^\infty[(-1)^jm_j\sigma^jT_1(j,\,\sigma,\,x)/j!]
\lesssim e^{-\alpha(\rho_0/\sigma)^{r_0}}1_{\{\infty\}}(S_{f_0})$ because
$\sum_{j=1}^\infty(m_j/j!)<\infty$.
We show that $\sum_{j=1}^\infty[(-1)^jm_j\sigma^jT_2(j,\,\sigma,\,x)/j!-c_j\sigma^j(f_0^{(j)}\ast \operatorname{sinc}_\sigma\ast\phi_\sigma)(x)]=0$ identically.
Algebra leads to $T_2(j,\,\sigma,\,x)=-\sum_{k=1}^\infty [m_{2k}\sigma^{2k}(f_0^{(j+2k)}\ast\operatorname{sinc}_\sigma\ast\phi_\sigma)(x)/(2k)!]$. Hence,
\[
\begin{split}
\sum_{j=1}^\infty\frac{(-1)^jm_j}{j!}\sigma^jT_2(j,\,\sigma,\,x)&
=-\sum_{j=1}^\infty\frac{m_{2j}}{(2j)!}\sum_{k=1}^\infty\frac{m_{2k}}{(2k)!}
\sigma^{2(j+k)}(f_0^{(2j+2k)}\ast\operatorname{sinc}_\sigma\ast\phi_\sigma)(x)\\
&=\sum_{s=2}^\infty c_{2s}\sigma^{2s}(f_0^{(2s)}\ast\operatorname{sinc}_\sigma\ast\phi_\sigma)(x)
\end{split}
\]
by definition of the numbers $c_{2s}$. The proof of \eqref{sup-norm} is thus complete.

Next, we prove \eqref{equality}. Because $T_1(j,\,\sigma,\,x)\lesssim\sigma^{-j}e^{-\alpha(\rho_0/\sigma)^{r_0}}1_{\{\infty\}}(S_{f_0})$ for $\sigma$ small enough,
$T_\sigma(f_0)=f_0-\sum_{j=1}^\infty d_j\sigma^jf_0^{(j)}+O(e^{-\alpha(\rho_0/\sigma)^{r_0}}1_{\{\infty\}}(S_{f_0}))$.
By definition of the $d_j$'s, taking into account that
$\sum_{j=1}^\infty[(-1)^jm_j\sigma^jf_0^{(j)}/j!]= f_0\ast\phi_\sigma-f_0$, we have
$f_0-\sum_{j=1}^\infty d_j\sigma^jf_0^{(j)}=f_0-(f_0\ast\phi_\sigma-f_0)-\sum_{j=1}^\infty c_j\sigma^jf_0^{(j)}=2
f_0-f_0\ast\phi_\sigma-\sum_{j=1}^\infty c_j\sigma^jf_0^{(j)}$, where
\[
\begin{split}
\sum_{j=2}^\infty c_{2j}\sigma^{2j}f_0^{(2j)}&=-\sum_{j=1}^\infty\frac{m_{2j}}{(2j)!}\sigma^{2j}
\sum_{k=1}^\infty
\frac{m_{2k}}{(2k)!}\sigma^{2k}f_0^{(2j+2k)}\\
&=-\sum_{j=1}^\infty\frac{m_{2j}}{(2j)!}\sigma^{2j}(f_0^{(2j)}\ast\phi_\sigma-f_0^{(2j)})\\
&=-(f_0\ast\phi_\sigma-f_0)\ast\phi_\sigma+(f_0\ast\phi_\sigma-f_0)\\
&=-f_0\ast\phi_\sigma\ast\phi_\sigma+2(f_0\ast\phi_\sigma)-f_0.
\end{split}
\]
Relationship \eqref{equality} follows.
To bound above $\int T_\sigma(f_0)\,\mathrm{d}\lambda$, note that
the coefficients in \eqref{equality} sum up to $1$.
Also, for $\sigma>0$ small enough, \[\int T_1(j,\,\sigma,\,x)1_{{\{z:\,T_1(j,\,\sigma,\,z)\neq0\}}}(x)\,\mathrm{d}x=
o(\sigma^{-j}e^{-\alpha(\rho_0/\sigma)^{r_0}})\] because
$\lim_{\sigma\rightarrow0}T_1(j,\,\sigma,\,x)=0$ identically so that
$\lim_{\sigma\rightarrow0}\lambda(T_1(j,\,\sigma,\,x)\neq0)=0$ for every $j\in\mathbb{N}$.
\end{proof}

Suppose that $f_0$ satisfies condition $(a)$.
Given $\delta\in(0,\,1)$, $c_1\in(0,\,\rho_0^{r_0}/2)$ and $B,\,M,\,\sigma>0$, let
\[
\begin{split}
B_\sigma&:=\{x\in\mathbb{R}:\,f_0(x)\geq B\sigma^{-M}e^{-c_1(1/\sigma)^{r_0}}\},\\
G_{\sigma}&:=\{x\in\mathbb{R}:\,T_\sigma(f_0)(x)>\delta f_0(x)\},\\
U_\sigma&:=\{x\in\mathbb{R}:\,|f_0^{(j)}(x)|\leq \sigma^{-j}f_0(x)/\sqrt{e}
,\,\,\,j\in\mathbb{N}\}.
\end{split}
\]
The function $T_\sigma(f_0)$ is modified to be non-negative by setting it
equal to a multiple of $f_0$ when it is below it. Let
$g_\sigma:=T_\sigma(f_0)1_{G_{\sigma}}+\delta f_01_{G^c_{\sigma}}$
be the modified function.
\begin{lem}\label{lem:intg}
Suppose that $f_0$ satisfies condition $(a)$ for some $r_0\in[1,\,2]$.
Let $\delta:=(1-\sqrt{e}/2)$. Then,
for $\sigma>0$ small enough, $\int g_\sigma\,\mathrm{d}\lambda\geq\delta$
and $\int g_\sigma\,\mathrm{d}\lambda=1+O(e^{-c_3(1/\sigma)^{r_0}})$ for a
suitable constant $c_3>0$.
\end{lem}
\begin{proof}
By definition, $g_\sigma\geq\delta f_0(1_{G_{\sigma}}+1_{G_{\sigma}^c})=\delta f_0$ so that
$\int g_\sigma\,\mathrm{d}\lambda\geq\delta$. Rewriting $g_\sigma=T_\sigma(f_0)+\pq{\delta f_0-T_\sigma(f_0)}1_{G_\sigma^c}$,
by \eqref{equality}, for $\sigma>0$ small enough, whatever $\alpha\in(0,\,1)$,
\[
\int g_\sigma\,\mathrm{d}\lambda=1+o(e^{-\alpha(\rho_0/\sigma)^{r_0}}1_{\{\infty\}}(S_{f_0}))+
\int[\delta f_0-T_\sigma(f_0)]1_{G_\sigma^c}\,\mathrm{d}\lambda=1+O(e^{-c_3(1/\sigma)^{r_0}}),
\]
since, for a suitable constant $c>0$,
\begin{equation}\label{integral}
\int[\delta f_0-T_\sigma(f_0)]1_{G_\sigma^c}\,\mathrm{d}\lambda=O(e^{-c(1/\sigma)^{r_0}}).
\end{equation}
To prove \eqref{integral}, we first show that, for $\sigma>0$ small enough,
$B_\sigma\cap U_\sigma\subseteq G_\sigma$. In effect, over the set $B_\sigma\cap U_\sigma$,
\[
\begin{split}
|T_\sigma(f_0)-f_0|&\leq f_0
\sum_{j=1}^\infty |d_j|\sigma^j(|f_0^{(j)}-f_0^{(j)}\ast\operatorname{sinc}_\sigma|/f_0)+
f_0\sum_{j=1}^\infty |d_j|\sigma^j(|f_0^{(j)}|/f_0)\\
&\leq f_0(e^{-1/2}+e^{-(\rho_0/\sigma)^{r_0}/2}/f_0)\sum_{j=1}^\infty |d_j|\\&\leq(\sqrt{e}-1)(1+O(\sigma^M))f_0
<[(\sqrt{e}-1)+(1-\sqrt{e}/2)]f_0=(\sqrt{e}/2)f_0,
\end{split}
\]
because $\sum_{j=1}^\infty |d_j|\leq (\sqrt{e}-1)\sqrt{e}$ and, over $B_\sigma$,
we have $e^{-(\rho_0/\sigma)^{r_0}/2}/f_0=O(\sigma^M)$. Hence,
$T_\sigma(f_0)>\delta f_0$ and $B_\sigma\cap U_\sigma\subseteq G_\sigma$.
Also, the set $U_\sigma^c$ has exponentially small probability.
By Markov's inequality, using the assumption that the sequence $(C_{0j}^{r_0/j})_{j\geq1}$ is bounded above,
for a suitable constant $k_3>0$,
\[
\begin{split}
\operatorname{P}_0(U_\sigma^c)&\leq\sum_{j=1}^\infty
\operatorname{P}_0(\exp{(|f_0^{(j)}(X)|/(C_{0j}f_0(X))^{r_0/j})}>\exp{((C_{0j}\sqrt{e})^{-r_0/j}(1/\sigma)^{r_0})})\\
&<e^{-k_3(1/\sigma)^{r_0}}\sum_{j=1}^\infty
\operatorname{E}_0[\exp{(|f_0^{(j)}(X)|/(C_{0j}f_0(X))^{r_0/j})}]\lesssim e^{-k_3(1/\sigma)^{r_0}}.
\end{split}
\]
Using the bounds
$\operatorname{P}_0(U_\sigma^c)\lesssim e^{-k_3(1/\sigma)^{r_0}}$,
$\operatorname{P}_0(B_\sigma^c)\lesssim(\sigma^{-M}e^{-c_1(1/\sigma)^{r_0}})^\gamma$ valid for every $\gamma\in(0,\,1)$, and the fact that, up to $O(e^{-\alpha(\rho_0/\sigma)^{r_0}}1_{\{\infty\}}(S_{f_0}))$,
the transform $T_\sigma(f_0)$ is a linear combination of $f_0$, $f_0\ast\phi_\sigma$ and
$f_0\ast\phi_\sigma\ast\phi_\sigma$, we prove that
$\int[\delta f_0-T_\sigma(f_0)]1_{B_\sigma^c\cup U_\sigma^c}\,\mathrm{d}\lambda\lesssim e^{-c(1/\sigma)^{r_0}}$.
We begin by showing that
$\int_{U_\sigma^c}(f_0\ast\phi_\sigma)\,\mathrm{d}\lambda\lesssim e^{-(k_3\wedge 2^{-1})(1/\sigma)^{r_0}}$. For
random variables $Y\sim f_0$ and $Z\sim\mathrm{N}(0,\,1)$,
$
\int_{U_\sigma^c}(f_0\ast\phi_\sigma)\,\mathrm{d}\lambda\leq\operatorname{P}(Y+\sigma Z\in U_\sigma^c,\,|Z|\leq\sigma^{-r_0/2})+\operatorname{P}(|Z|>\sigma^{-r_0/2})=:T_1+T_2,
$
where $T_2\lesssim e^{-(1/\sigma)^{r_0}/2}$ and $T_1\leq P_0(U_\sigma^c)\lesssim e^{-k_3(1/\sigma)^{r_0}}$.
By the result just shown,
$\int_{B_\sigma^c}(f_0\ast\phi_\sigma)\,\mathrm{d}\lambda
\lesssim\int_{B_\sigma^c\cap U_\sigma}(f_0\ast\phi_\sigma)\,\mathrm{d}\lambda+
e^{-(k_3\wedge 2^{-1})(1/\sigma)^{r_0}}$, where, for $\xi>1$,
\[
\begin{split}
\int_{B_\sigma^c\cap U_\sigma}(f_0\ast\phi_\sigma)\,\mathrm{d}\lambda
&\leq\operatorname{P}(Y+\sigma Z\in B_\sigma^c\cap U_\sigma,\,|Z|\leq \sigma^{-r_0/2},\,Y\in B_{\xi\sigma}\cap U_\sigma)
\\&\qquad+\operatorname{P}(Y\in U_\sigma^c)+P(Y\in B^c_{\xi\sigma})+\operatorname{P}(|Z|>\sigma^{-r_0/2})\lesssim e^{-k_4(1/\sigma)^{r_0}}.
\end{split}
\]
Analogously, $\int_{B_\sigma^c\cup U_\sigma^c}(f_0\ast\phi_\sigma\ast\phi_\sigma)\,\mathrm{d}\lambda\lesssim e^{-c'(1/\sigma)^{r_0}}$,
which completes the proof.
\end{proof}

Next, a finite Gaussian mixture, denoted $m_\sigma$,
is constructed from the re-normalized restriction to a compact set of
the density derived from $g_\sigma$ such that it still approximates $f_0$, in the Kullback-Leibler divergence,
with an error of the order $O(e^{-c(1/\sigma)^{r_0}})$.
\begin{lem}\label{appro}
Suppose that $f_0$ satisfies conditions $(a)$ for some $r_0\in[1,\,2]$, $(b)$ and $(c)$.
For $\sigma>0$ small enough, there exists a finite Gaussian
mixture $m_\sigma$, having at most $N_\sigma=O((a_\sigma/\sigma)^2)$ support points
in $[-a_\sigma,\,a_\sigma]$, with $a_\sigma=O(\sigma^{-r_0/(\varpi\wedge 2)})$, such that,
for finite suitable constants $S,\,c_5>0$,
\begin{equation}\label{moments}
\max\{\operatorname{KL}(f_0;\,m_\sigma),\,\operatorname{E}_0[(\log(f_0/m_\sigma))^2]\}\lesssim \sigma^{-S}e^{-c_5(1/\sigma)^{r_0}}.
\end{equation}
\end{lem}
\begin{proof}
We give the proof only for the bound on the Kullback-Leibler divergence, which is decomposed into the sum of three integrals, see \eqref{split} below. We begin by bounding the first integral. Fix $\zeta\in(0,\,1)$ and
let $C_\zeta>0$ be the same constant appearing in Lemma~\ref{lem:lowerboundf0}.
Choose $\delta:=(1-\sqrt{e}/2)\in(0,\,1)$.
Set $C_{g_\sigma}:=\int g_\sigma\,\mathrm{d}\lambda$,
by Lemma~\ref{lem:intg}, for $\sigma>0$ small enough,
$C_{g_\sigma}=1+Ae^{-c_3(1/\sigma)^{r_0}}$ for a suitable positive constant $A$.
Defined the density $h_\sigma:=g_\sigma/C_{g_\sigma}$,
\[
\forall\,\sigma<\tau_\zeta,\qquad h_\sigma\ast\phi_\sigma
\geq\frac{\delta(f_0\ast\phi_\sigma)}{1+Ae^{-c_3(1/\sigma)^{r_0}}}
\geq\frac{\delta C_\zeta}{1+Ae^{-c_3(1/\sigma)^{r_0}}}f_0,
\]
because $g_\sigma\geq\delta f_0$ and Lemma~\ref{lem:lowerboundf0} applies.
Furthermore, $|h_\sigma\ast\phi_\sigma-f_0|\leq C_{g_\sigma}^{-1}|g_\sigma\ast\phi_\sigma-f_0|+|C_{g_\sigma}^{-1}-1|f_0
\lesssim|g_\sigma\ast\phi_\sigma-f_0|+e^{-c_3(1/\sigma)^{r_0}}f_0$.
Lemma~\ref{approxlem} and the inequality $\int[\delta f_0-T_\sigma(f_0)]1_{G_\sigma^c}\,\mathrm{d}\lambda
\leq\int[\delta f_0-T_\sigma(f_0)]1_{B_\sigma^c\cup U_\sigma^c}\,\mathrm{d}\lambda\lesssim e^{-c(1/\sigma)^{r_0}}$
imply that, for $\sigma>0$ small enough, whatever $\alpha\in(0,\,1)$,
\[
\begin{split}
|g_\sigma\ast\phi_\sigma-f_0|&\leq
|T_\sigma(f_0)\ast\phi_\sigma -f_0|+|[(\delta f_0-T_\sigma(f_0))1_{G_\sigma^c}]\ast\phi_\sigma|\\
&\lesssim e^{-\alpha(\rho_0/\sigma)^{r_0}}1_{\{\infty\}}(S_{f_0})
+\sigma^{-1}e^{-c(1/\sigma)^{r_0}}.
\end{split}
\]
Therefore,
\begin{equation}\label{LBhp}
\|h_\sigma\ast\phi_\sigma-f_0\|_\infty\lesssim e^{-c_4(1/\sigma)^{r_0}},
\end{equation}
where $c_4:=\min\{\alpha\rho^{r_0}_0,\,c_3,\,\alpha c\}$.
Now, $\operatorname{KL}(f_0;\,h_\sigma\ast\phi_\sigma)=
(\int_{B_\sigma\cap U_\sigma}+\int_{B_\sigma^c\cup U_\sigma^c})f_0\log(f_0/(h_\sigma\ast\phi_\sigma))\,\mathrm{d}\lambda=:I_1+I_2$.
For $c_4>c_1>0$, by inequality~\eqref{LBhp},
\[
\begin{split}
I_1&\leq\frac{\sup_{x\in B_\sigma\cap U_\sigma}|f_0(x)-(h_\sigma\ast\phi_\sigma)(x)|}{\inf_{x\in B_\sigma} f_0(x)-\sup_{x\in B_\sigma\cap U_\sigma}|f_0(x)-(h_\sigma\ast\phi_\sigma)(x)|}\int_{B_\sigma\cap U_\sigma}f_0\,\mathrm{d}\lambda\\
&\lesssim\frac{e^{-c_4(1/\sigma)^{r_0}}}{e^{-c_1(1/\sigma)^{r_0}}
(B\sigma^{-M}- De^{-(c_4-c_1)(1/\sigma)^{r_0}})}
\lesssim e^{-(c_4-c_1)(1/\sigma)^{r_0}}.
\end{split}
\]
By Lemma~\ref{lem:intg}, for every $\gamma\in(0,\,1)$, $\int_{B_\sigma^c\cup U_\sigma^c}f_0\,\mathrm{d}\lambda\lesssim
(\sigma^{-M}e^{-c_1(1/\sigma)^{r_0}})^\gamma+e^{-k_3(1/\sigma)^{r_0}}$. Therefore,
$
I_2\lesssim(\sigma^{-\gamma M}e^{-\gamma c_1(1/\sigma)^{r_0}}+e^{-k_3(1/\sigma)^{r_0}})\log
((1+A e^{-c_3(1/\sigma)^{r_0}})/(\delta C_\zeta)),
$
where the logarithmic term is positive because $0<\delta C_\zeta<1$. Thus,
\[
\mathrm{KL}(f_0;\,h_\sigma\ast\phi_\sigma)\lesssim
\sigma^{-\gamma M}e^{-\min\{(c_4-c_1),\,\gamma c_1,\,k_3\}
(1/\sigma)^{r_0}}.
\]
Next, let $C_{h_\sigma}:=\int_{-a_\sigma}^{a_\sigma}h_\sigma\,\mathrm{d}\lambda$ and
define $\tilde{h}_\sigma:= h_\sigma1_{[-a_\sigma,\,a_\sigma]}/C_{h_\sigma}$ as
the re-normalized restriction of $h_\sigma$ to $[-a_\sigma,\,a_\sigma]$.
By Lemma~\ref{Npoints},
there exists a discrete distribution $\tilde{F}$ on
$[-a_\sigma,\,a_\sigma]$, with at most $N_\sigma=O((a_\sigma/\sigma)^2)$ support points, such that
$\|\tilde{h}_\sigma\ast\phi_\sigma-\tilde{F}\ast\phi_\sigma\|_\infty\lesssim\sigma^{-1}e^{-N_\sigma}$.
Set $\widetilde{m}_\sigma:=C_{h_\sigma}(\tilde{F}\ast\phi_\sigma)$, we have
$|h_\sigma\ast\phi_\sigma-\widetilde{m}_\sigma|
\leq \sigma^{-1}e^{-N_\sigma}+(h_\sigma1_{[-a_\sigma,\,a_\sigma]^c})\ast\phi_\sigma$.
For $\sigma>0$ small enough, we have $(h_\sigma1_{[-a_\sigma,\,a_\sigma]^c})\ast\phi_\sigma
\lesssim e^{-(\rho_0/\sigma)^{r_0}/2}1_{\{\infty\}}(S_{f_0})+e^{-c_0(a_\sigma/2)^{\varpi\wedge 2}}$
in virtue of Lemma~\ref{approxlem} and assumption $(c)$ on $f_0$.
Thus, for a constant $c''>0$ such that $c_1<c''\leq[(\rho_0^{r_0}/2)\wedge(c_0/2^{(\varpi\wedge 2)})]$,
\[\|h_\sigma\ast\phi_\sigma-\widetilde{m}_\sigma\|_\infty\lesssim \sigma^{-1}e^{-N_\sigma}+e^{-(\rho_0/\sigma)^{r_0}/2}1_{\{\infty\}}(S_{f_0})+e^{-c_0(a_\sigma/2)^{\varpi\wedge 2}}\lesssim
e^{-c''(1/\sigma)^{r_0}}.\]
Let $t:=\widetilde{m}_\sigma+D_\sigma\phi_\sigma$,
with $D_\sigma:=\sigma^{-(R-1)}e^{-\tilde{c}(1/\sigma)^{r_0}}$ for $1<R<M$ and $\tilde{c}>c_1$.
Define the finite Gaussian mixture $m_\sigma:=(\int t\,\mathrm{d}\lambda)^{-1}t=
(\widetilde{m}_\sigma+D_\sigma\phi_\sigma)/(C_{h_\sigma}+D_\sigma)$.
Write
\begin{eqnarray}\label{split}
\hspace*{-1.2cm}\mathrm{KL}(f_0;\,m_\sigma)&=&
\int f_0\log\frac{f_0}{h_\sigma\ast \phi_\sigma}\,\mathrm{d}\lambda
+\int f_0\log\frac{h_\sigma\ast\phi_\sigma}{t}\,\mathrm{d}\lambda+\int f_0\log\frac{t}{m_\sigma}\,\mathrm{d}\lambda\\
&=:&J_1+J_2+J_3,\nonumber
\end{eqnarray}
where $J_1=\operatorname{KL}(f_0;\,h_\sigma\ast\phi_\sigma)$.

$\bullet$ \emph{Control of $J_1$}. It has already been shown
that $J_1\lesssim\sigma^{-\gamma M}e^{-\min\{(c_4-c_1),\,\gamma c_1,\,k_3\}
(1/\sigma)^{r_0}}$.

$\bullet$ \emph{Control of $J_2$}. Write
$J_2=(\int_{B_\sigma}+\int_{B^c_\sigma})f_0\log((h_\sigma\ast\phi_\sigma)/t)\,\mathrm{d}\lambda=:
J_{21}+J_{22}$. Since $0<c_1<(c''\wedge\tilde{c})$,
\[
\begin{split}
J_{21}&\leq\int_{B_\sigma}f_0\frac{h_\sigma\ast\phi_\sigma-t}{t}\,\mathrm{d}\lambda\\
&\lesssim\frac{\sigma^{-R}e^{-(c''\wedge\tilde{c})(1/\sigma)^{r_0}}}
{B\sigma^{-M}e^{-c_1(1/\sigma)^{r_0}}-
e^{-c''(1/\sigma)^{r_0}}}\int_{B_\sigma}f_0\,\mathrm{d}\lambda\\
&\lesssim\sigma^{M-R}e^{-[(c''\wedge\tilde{c})-c_1](1/\sigma)^{r_0}},
\end{split}
\]
because $
|h_\sigma\ast\phi_\sigma-t|\leq|h_\sigma\ast\phi_\sigma-\widetilde{m}|+D_\sigma\phi_\sigma
\lesssim\sigma^{-R}e^{-(c''\wedge\tilde{c})(1/\sigma)^{r_0}}$ and, over $B_\sigma$, $h_\sigma\ast\phi_\sigma\gtrsim f_0\gtrsim B\sigma^{-M}e^{-c_1(1/\sigma)^{r_0}}$ so that
$t>\widetilde{m}_\sigma\geq h_\sigma\ast\phi_\sigma-|h_\sigma\ast\phi_\sigma-\widetilde{m}_\sigma|
\gtrsim \sigma^{-M}e^{-c_1(1/\sigma)^{r_0}}-
e^{-c''(1/\sigma)^{r_0}}$. Because $\|h_\sigma\ast\phi_\sigma\|_\infty\leq C_0<\infty$ for a
constant $C_0$ (possibly depending on $f_0$) and
$t\geq D_\sigma\phi_\sigma$,
\[
\begin{split}
J_{22}&\lesssim\log(\sigma/D_\sigma)\int_{B^c_\sigma}f_0\,\mathrm{d}\lambda
+\frac{1}{2\sigma^{2}}\int_{B^c_\sigma}x^2f_0(x)\,\mathrm{d}x\\
&\lesssim\sigma^{-(\gamma M+r_0)}e^{-\gamma c_1(1/\sigma)^{r_0}}+\sigma^{-(\gamma M+2)}e^{-\gamma c_1(1/\sigma)^{r_0}}\lesssim\sigma^{-[\gamma M+(r_0\vee 2)]}e^{-\gamma c_1(1/\sigma)^{r_0}}.
\end{split}
\]

$\bullet$ \emph{Control of $J_3$}. Noting that $t/m_\sigma=
C_{h_\sigma}+D_\sigma\leq 1+D_\sigma$, we have
$J_3\leq\log(1+D_\sigma)\leq D_\sigma=
\sigma^{-(R-1)}e^{-\tilde{c}(1/\sigma)^{r_0}}$.

Combining partial results, for $0<c_1<\min\{c_4,\,c'',\,\tilde{c}\}$, we have
$\operatorname{KL}(f_0;\,m_\sigma)\lesssim\sigma^{-S}e^{-c_5(1/\sigma)^{r_0}}$,
where $S\geq\max\{M-R,\,R-1,\,\gamma M+(r_0\vee 2)\}$ and
$c_5:=\min\{\tilde{c},\,\gamma c_1,\,(c''\wedge\tilde{c})-c_1,\,k_3,\,c_4-c_1\}$
are finite constants. The same reasoning applies to
$\operatorname{E}_0[(\log(f_0/m_\sigma))^2]$ and
\eqref{moments} follows.
\end{proof}
\begin{proof}[Proof of Theorem~\ref{analytic}]
The proof is presented for the case where $r_0\in[1,\,2]$.
For the case where $r_0>2$, the assertion holds with $(r_0\wedge 2)=2$.
As in the proof of Theorem~\ref{superKs}, we first show the result for the $L^1$-metric.
Then, we deal with $L^p$-metrics, $p\in[2,\,\infty]$. The case of $L^p$-metrics,
$p\in(1,\,2)$, is covered by interpolation.

$\bullet$ \emph{$L^1$-metric}. Proceeding as in the proof of Theorem~\ref{superKs},
since $2\psi(r_0,\,d)>1$ for every $d\in[0,\,1)$, we have
$\varepsilon_{n,1}:=(\bar{\varepsilon}_n\vee\tilde{\varepsilon}_n)=\bar{\varepsilon}_n=n^{-1/2}(\log n)^{\frac{1}{2}+\{\frac{1}{2}\vee[2(1+\frac{1}{\delta})\psi(r_0,\,d)]\}}$,
with $\psi(r_0,\,d)$ defined as in \eqref{power}.

$\bullet$ \emph{$L^p$-metrics, $p\in[2,\,\infty]$}. Conditions of
Theorem~\ref{ThsuperKs} are satisfied. Let $\varepsilon_{n,p}:=\tilde{\varepsilon}_n(n\tilde{\varepsilon}_n^2)^{(1-1/p)/2}$.
By the assumption that $f_0\in \mathcal{A}^{\rho_0,\,r_0,\,L_0}(\mathbb{R})$,
in virtue of Lemma~\ref{lem:admis-approx-seq}, for every $p\in[2,\,\infty]$,
letting $2^{J_n}=cn\tilde{\varepsilon}_n^2$, with $c$ defined as in the proof
of Theorem~\ref{ThsuperKs} ($r=2$ and $\rho=2^{-1/2}$ for the Gaussian kernel),
for $n$ large enough, whatever $\alpha\in(0,\,1)$, we have
$\|f_0\ast \operatorname{sinc}_{2^{-J_n}}-f_0\|_p\lesssim \exp{(-\alpha(\rho_0c)^{r_0}(n\tilde{\varepsilon}_n^2)^{r_0})}\lesssim n^{-1}\lesssim\varepsilon_{n,p}$
because $2r_0\psi(r_0,\,d)>1$ for every $d\in[0,\,1)$.

$\bullet$ \emph{Small ball probability estimate}. We show that, for a suitable constant $c_2>0$,
$(\Pi\times G)(B_{\operatorname{KL}}(f_0;\,\tilde{\varepsilon}_n^2))\gtrsim
\exp(-c_2n\tilde{\varepsilon}_n^2)$,
with $\tilde{\varepsilon}_n=n^{-1/2}(\log n)^{\psi(r_0,\,d)}$.
By Lemma~\ref{appro}, for $\sigma>0$ small enough,
there exists a finite Gaussian mixture $m_\sigma$, with $N_\sigma=O((a_\sigma/\sigma)^2)$ support points
$\theta_1,\,\ldots,\,\theta_{N_\sigma}$ in $[-a_\sigma,\,a_\sigma]$,
where $a_\sigma=O(\sigma^{-r_0/(\varpi\wedge 2)})$, such that \eqref{moments} holds.
Let $p_1,\,\ldots,\,p_{N_\sigma}$ denote the mixing weights of $m_\sigma$.
The inequality in \eqref{moments} holds for any Gaussian mixture
$m_{\sigma'}$, with $\sigma'\in[\sigma,\,\sigma+e^{-d_1(1/\sigma)^{r_0}})$,
having support points $\theta'_1,\,\ldots,\,\theta'_{N_{\sigma'}}$
such that $\sum_{j=1}^{N_{\sigma'}}|\theta'_j-\theta_j|\leq e^{-d_2(1/\sigma)^{r_0}}$ and mixing weights
$p'_1,\,\ldots,\,p'_{N_{\sigma'}}$ such that $\sum_{j=1}^{N_{\sigma'}}|p'_j-p_j|\leq e^{-d_3(1/\sigma)^{r_0}}$ for suitable constants $d_1,\,d_2,\,d_3>0$. Let $\tilde{B}_\sigma:=\{f_0\geq\zeta_\sigma\}$, with
$\zeta_\sigma:=B'\sigma^{-S'}
e^{-c(1/\sigma)^{r_0}}$, where $S':=(S-2)/\omega$, with $\frac{1}{2}<\omega<1$ arbitrarily fixed
and $c_1<c<3c_5$, the constants $S>2$, $c_1$ and $c_5$ being those appearing in Lemma~\ref{appro}.
For any $F\in\mathscr{M}(\mathbb{R})$ and
$\sigma'\in[\sigma,\,\sigma+e^{-d_1(1/\sigma)^{r_0}})$, $\operatorname{KL}(f_0;\,f_{F,\,\sigma'})
\lesssim\sigma^{-S}e^{-c_5(1/\sigma)^{r_0}}+(
\int_{\tilde{B}_\sigma}+\int_{\tilde{B}_\sigma^c})f_0\log(m_{\sigma'}/f_{F,\,\sigma'})\,\mathrm{d}\lambda$.
We begin by providing an upper bound on the second integral.
For any $F$ such that $F([-a_{\sigma'},\,a_{\sigma'}])\geq\frac{1}{2}$,
we have $f_{F,\,\sigma'}(x)\gtrsim
(\sigma')^{-1}\exp{(-(x^2+a_{\sigma'}^2)/(\sigma')^2)}$ for all $x\in\mathbb{R}$.
From Lemma~\ref{appro},
$\|m_{\sigma'}\|_\infty\lesssim{(\sigma')}^{-1}$.
Also, $\int_{\tilde{B}_\sigma^c}(x/{\sigma'})^2f_0(x)\,\mathrm{d}x
\lesssim\sigma^{-2}\zeta_\sigma^{\omega}$
and $\int_{\tilde{B}_\sigma^c}f_0\,\mathrm{d}\lambda\lesssim
\zeta_\sigma^{\omega}$. Therefore, for a suitable constant $c'>0$,
$
\int_{\tilde{B}_\sigma^c} f_0\log(m_{\sigma'}/f_{F,\,\sigma'})\,\mathrm{d}\lambda
\lesssim\int_{\tilde{B}_\sigma^c}(x/\sigma')^2f_0(x)\,\mathrm{d}x
+(a_{\sigma'}/\sigma')^2\int_{\tilde{B}_\sigma^c}f_0\,\mathrm{d}\lambda
\lesssim e^{-c'(1/\sigma)^{r_0}}$.

Next, as in the proof of
Theorem~\ref{superKs}, we distinguish the case where the prior for $F$ is
a Dirichlet or a N-IG process, from the case where the prior for $F$
is a general Pitman-Yor process with $d\in[0,\,1)$ and $c>-d$.

$-$ \emph{Dirichlet or N-IG process}. Clearly, $\int_{\tilde{B}_\sigma}f_0\log(m_{\sigma'}/f_{F,\,\sigma'})\,\mathrm{d}\lambda\leq \int_{\tilde{B}_\sigma}f_0
(\|m_{\sigma'}-f_{F,\,\sigma'}\|_\infty/f_{F,\,\sigma'})\,\mathrm{d}\lambda.$
Using Lemma~5 of Ghosal and van der Vaart~\cite{GvdV072}, page~711, we get $\|m_{\sigma'}-f_{F,\,\sigma'}\|_\infty
\lesssim\sigma^{-2}\max_{1\leq j\leq N_{\sigma'}}\lambda(U_j)+\sigma^{-1}\sum_{j=1}^{N_{\sigma'}}|F(U_j)-p_j|$,
where $U_0,\,\ldots,\,U_{N_{\sigma'}}$ is a partition of
$\mathbb{R}$, with $U_0:=(\bigcup_{j=1}^{N_{\sigma'}}U_j)^c$ and
$U_j\ni \theta_j$ for $j=1,\,\ldots,\,N_{\sigma'}$.
The support points of $m_{\sigma'}$ can be taken to be
at least $\sigma^{-3(S-2)}e^{-3c_5(1/\sigma)^{r_0}}$-separated. If not, $m_{\sigma'}$ can be projected onto a
mixture $m'_{\sigma'}$, with $\sigma^{-3(S-2)}e^{-3c_5(1/\sigma)^{r_0}}$-separated points, such that
$\|m_{\sigma'}-m'_{\sigma'}\|_\infty\lesssim\sigma^{-(3S-4)}e^{-3c_5(1/\sigma)^{r_0}}$.
Thus, we can find disjoint intervals $U_1,\,\ldots,\,U_{N_{\sigma'}}$ such that $U_j\ni\theta_j$
and $\sigma^{-3(S-2)}e^{-3c_5(1/\sigma)^{r_0}}\leq\lambda(U_j)
\leq2\sigma^{-3(S-2)}e^{-3c_5(1/\sigma)^{r_0}}$, $j=1,\,\ldots,\,N_{\sigma'}$.
Let $F$ be such that
\begin{equation}\label{probmass}
\sum_{j=1}^{N_{\sigma'}}|F(U_j)-p_j|\leq\sigma^{-(3S-5)}e^{-3c_5(1/\sigma)^{r_0}}.
\end{equation}
Then,
$\|m_{\sigma'}-f_{F,\,\sigma'}\|_\infty\lesssim\sigma^{-(3S-4)}e^{-3c_5(1/\sigma)^{r_0}}$ and, over the set $\tilde{B}_{\sigma}$,
$f_{F,\,\sigma'}\gtrsim m_{\sigma'}-\sigma^{-(3S-4)}e^{-3c_5(1/\sigma)^{r_0}}\gtrsim\zeta_\sigma$.
Therefore,
$\int_{\tilde{B}_\sigma}f_0\log(m_{\sigma'}/f_{F,\,\sigma'})\,\mathrm{d}\lambda\lesssim \sigma^{-S}e^{-\alpha(3c_5-c)(1/\sigma)^{r_0}}$.
Note that, for $F$ satisfying (\ref{probmass}), $F([-a_{\sigma'},\,a_{\sigma'}])\geq\frac{1}{2}$.
Combining partial results, $\max\{\operatorname{KL}(f_0;\,f_{F,\,\sigma'}),\,\operatorname{E}_0[(\log(f_0/f_{F,\,\sigma'}))^2]\}
\lesssim\sigma^{-S}e^{-c_6(1/\sigma)^{r_0}}$
for $0<c_6\leq\min\{\alpha(3c_5-c),\,c'\}$.
In order to apply Lemma~A.2 of Ghosal and van der
Vaart~\cite{GvdV01}, pages~1260--1261, to estimate the prior
probability of $\{F:\,\sum_{j=1}^{N_{\sigma'}}|F(U_j)-p_j|\leq\sigma^{-(3S-5)}e^{-3c_5(1/\sigma)^{r_0}}\}$, note that
$\alpha(U_j)\geq \lambda(U_j)\inf_{|\theta|\leq a_{\sigma'}}
\alpha'(\theta)\gtrsim
\sigma^{-3(S-2)}e^{-(3c_5+b)(1/\sigma)^{r_0}}$
because $\delta\in(0,\,2]$. Also,
$N_{\sigma'}\sigma^{-(3S-5)}e^{-3c_5(1/\sigma)^{r_0}}\lesssim 1$. Therefore, since $r_0\geq1$,
\[\begin{split}
(\Pi\times G)(B_{\operatorname{KL}}(f_0;\,\sigma^{-S}e^{-c_6(1/\sigma)^{r_0}}))\\
&\hspace*{-4.9cm}
\gtrsim
\operatorname{P}([\sigma,\,\sigma+e^{-d_1(1/\sigma)^{r_0}}))\times
\operatorname{P}\pt{\sum_{j=1}^{N_{\sigma'}}|F(U_j)-p_j|\leq\sigma^{-(3S-5)}e^{-3c_5(1/\sigma)^{r_0}}}\\
&\hspace*{-4.9cm}\gtrsim\exp{(-D_1(1/\sigma)(\log n)^t-(1/\sigma)^{r_0}(d_1-c_7N_{\sigma'}))}\gtrsim\exp{(-c_8(1/\sigma)^{r_0}[(\log n)^t\vee N_\sigma])}
\end{split}
\]
for a suitable constant $c_8>0$. Taking $\sigma\equiv \sigma_n=O((\log n)^{-1/r_0})$, we have
$(1/\sigma)^{r_0}[(\log n)^t\vee N_\sigma]\lesssim (\log n)^{2\psi(r_0,\,0)}$. Therefore,
we need to take $S=2r_0\psi(r_0,\,0)$, while having $S>2$ and
$S\geq\max\{M-R,\,R-1,\,\gamma M+(r_0\vee 2)\}$, as prescribed in Lemma~\ref{appro}.
Since $2r_0\psi(r_0,\,0)>2$, the latter constraint is met
by suitably choosing $M$ and $R$.

$-$ \emph{Pitman-Yor process with $d\in[0,\,1)$ and $c>-d$}. It is enough to note that
$\|m_{\sigma'}-f_{F,\,\sigma'}\|_\infty\lesssim
{\sigma}^{-1}\sum_{j=1}^{M_{\sigma'}}|W_j-p_j|+\sigma^{-2}\sum_{j=1}^{M_{\sigma'}}p_j|Z_j-\theta_j|$
and then proceed estimating the probabilities in $a)$ and $b)$ of Theorem~\ref{superKs}.
Thus,
\[(\Pi\times G)(B_{\operatorname{KL}}(f_0;\,\sigma^{-S}e^{-c_6(1/\sigma)^{r_0}}))
\gtrsim\exp{(-c_8(1/\sigma)^{r_0}[(\log n)^t\vee N^2_\sigma])}.
\]
Again, taking $\sigma\equiv \sigma_n=O((\log n)^{-1/r_0})$, we have
$(1/\sigma)^{r_0}[(\log n)^t\vee N^2_\sigma]\lesssim (\log n)^{2\psi(r_0,\,d)}$.
So, $S=2r_0\psi(r_0,\,d)$ and the constraints $S>2$ and
$S\geq\max\{M-R,\,R-1,\,\gamma M+(r_0\vee 2)\}$ are met
by properly choosing $M$ and $R$.
\end{proof}


\section[Appendix]{Appendix}\label{App}

Subsection~\ref{appendixA} reports the arguments for the results in Section~\ref{sec:second}.
Subsection~\ref{ana-loc} contains the proof of
Theorem~\ref{superKs}.
Subsection~\ref{ada-sobo} reports the proof of Theorem~\ref{adaptiveSobolev}.
Subsection~\ref{AppendixC} reports auxiliary results.

\subsection[Proofs of the results in Section~\ref{sec:second}]{Proofs of the results in Section~\ref{sec:second}}\label{appendixA}

The following lemma provides an upper bound on the
$L^p$-norm approximation error of a density, whose
Fourier transform either vanishes outside a compact or
decays exponentially fast, by its convolution with the $\operatorname{sinc}$ kernel.
For any probability density $f$, define
the positive (possibly infinite) constant $S_f:=\sup\{|t|:|\hat{f}(t)|\neq0\}$. If
\begin{itemize}
\item $S_f<\infty$, then $\mathrm{supp}(|\hat{f}|)\subseteq[-S_f,\,S_f]$,
\item $S_f=\infty$, then $|\hat{f}|>0$ everywhere.
\end{itemize}
If $\hat{f}\in L^1(\mathbb{R})$,
then $f$ can be recovered from $\hat{f}$ using the inversion formula
$f(x)={(2\pi)}^{-1}\int e^{-itx}\hat{f}(t)\,\mathrm{d}t$,
$x\in\mathbb{R}$. Furthermore, $f$ is continuous and bounded.
\begin{lem}\label{lem:admis-approx-seq}
Let $f\in\mathcal{A}^{\rho,\,r,\,L}(\mathbb{R})$ for some constants
$\rho,\,r,\,L>0$. Let $\sigma>0$ be fixed. If $S_f\leq 1/\sigma$, then
$\|f\ast \operatorname{sinc}_\sigma-f\|_p=0$ for every $p\in[1,\,\infty]$.
If $S_f=\infty$,
then, for any $\alpha\in(0,\,1)$, we have $\|f\ast\operatorname{sinc}_\sigma-f\|_p\lesssim e^{-\alpha(\rho/\sigma)^{r}}$ for every $p\in[2,\,\infty]$.
\end{lem}
\begin{proof}
By the inversion formula and the fact that $\widehat{\operatorname{sinc}}(t)=1_{[-1,\,1]}(t)$, $t\in\mathbb{R}$,
we have $(f\ast\operatorname{sinc}_{\sigma}-f)(x)=(2\pi)^{-1}\int_{|t|>1/\sigma}e^{-itx}\hat{f}(t)\,\mathrm{d}t$, $x\in\mathbb{R}$. If $S_f\leq1/\sigma$, then $\int_{|t|>1/\sigma}e^{-itx}\hat{f}(t)\,\mathrm{d}t=0$ identically and
$\|f\ast \operatorname{sinc}_\sigma-f\|_p=0$ for every $p\in[1,\,\infty]$.
Next, suppose $S_f=\infty$. For any function $g\in L^p(\mathbb{R})$, $p\in[2,\,\infty)$, we have
$\|g\|_p^p\leq C_p \|\hat{g}\|_q^q$, where $q^{-1}:=(1-p^{-1})\in[1/2,\,1)$
and $C_p>0$ is a constant depending only on $p$, see, \emph{e.g.},
Theorem~74 in Titchmarsh~\cite{Titchmarsh1937}, page~96. By the assumption that $f\in\mathcal{A}^{\rho,\,r,\,L}(\mathbb{R})$, we have $f\in L^p(\mathbb{R})$
for every $p\in[2,\,\infty]$. Thus, for every $p\in[2,\,\infty)$, we have
$\|f\ast \operatorname{sinc}_{\sigma}-f\|_p\leq \|f\|_1\|\operatorname{sinc}_{\sigma}\|_p+\|f\|_p<\infty$ and $\|f\ast \operatorname{sinc}_{\sigma}-f\|_p^p\leq C_p\|\hat{f}(\widehat{\operatorname{sinc}_\sigma}-1)\|_q^q
=C_p\int_{|t|>1/\sigma}|\hat{f}(t)|^q\,\mathrm{d}t$.
By the Cauchy-Schwarz inequality and the assumption that $f\in\mathcal{A}^{\rho,\,r,\,L}(\mathbb{R})$, for any $\alpha\in(0,\,1)$,
\begin{equation}\label{infinity}
\int_{|t|>1/\sigma}|\hat{f}(t)|^q\,\mathrm{d}t\leq\int_{|t|>1/\sigma}|\hat{f}(t)|\,\mathrm{d}t
\lesssim\sigma^{-(1-r)/2}e^{-(\rho/\sigma)^{r}}
\lesssim e^{-\alpha(\rho/\sigma)^{r}},
\end{equation}
where $\int_{1/\sigma}^\infty e^{-2(\rho t)^{r}}\,\mathrm{d}t={r}^{-1}(2\rho^r)^{-1/r}\Gamma(r^{-1},\,2(\rho/\sigma)^r)$,
with $\Gamma(a,\,z)=\int_z^\infty t^{a-1}e^{-t}\,\mathrm{d}t$, for $a,\,z>0$,
the upper incomplete gamma function. It is known that
$\Gamma(a,\,z)\sim z^{a-1}e^{-z}$ as $z\rightarrow\infty$.
The case where $p=\infty$ is treated implicitly in \eqref{infinity}.
\end{proof}
When $S_f=\infty$, the result can be extended to all $L^p$-metrics, $p\in[1,\,\infty]$, replacing the $\operatorname{sinc}$ kernel
with a superkernel, which, unlike the $\operatorname{sinc}$ kernel, is an absolutely integrable function.
In fact, by definition, a \emph{superkernel} $S$ is a symmetric, absolutely integrable function with $\int S\,\mathrm{d}\lambda=1$, having
absolutely integrable Fourier transform $\hat{S}$ (hence $S$ is bounded), with the properties that
$\hat{S}=1$ identically on $[-1,\,1]$ and $|\hat{S}|<1$ outside $[-1,\,1]$. The interval $[-1,\,1]$
is chosen for convenience only: $\hat{S}$ is required to be equal to $1$ in a neighborhood of $0$. Superkernels necessarily have infinite support. They can be obtained as iterated convolutions of re-scaled versions of the $\operatorname{sinc}$ kernel, cf. Example 1 in Devroye~\cite{Devroye92}, page~2039.
\begin{lem}\label{lem:admis-approx-seq2}
Let $f\in\mathcal{A}^{\rho,\,r,\,L}(\mathbb{R})$ for some constants
$\rho,\,r,\,L>0$. Let $S$ be a superkernel and $\sigma>0$ be fixed. If $S_f\leq 1/\sigma$, then
$\|f\ast S_\sigma-f\|_p=0$ for every $p\in[1,\,\infty]$.
If $S_f=\infty$,
then, for any $\alpha\in(0,\,1)$, we have $\|f\ast S_\sigma-f\|_p\lesssim e^{-\alpha(\rho/\sigma)^{r}}$ for every $p\in[2,\,\infty]$.
If, furthermore, when $S_f=\infty$,
for some $\upsilon\in(0,\,1)$, we have $\int f^\upsilon\,\mathrm{d}\lambda<\infty$, then
$\|f\ast S_\sigma-f\|_p\lesssim e^{-\alpha(1-\upsilon)(\rho/\sigma)^{r}}$ for every $p\in[1,\,2)$.
\end{lem}
\begin{proof}
We have $(f\ast S_{\sigma}-f)(x)={(2\pi)}^{-1}\int_{|t|>1/\sigma}e^{-itx}\hat{f}(t)[\hat{S}(\sigma t)-1]\,\mathrm{d}t$, $x\in\mathbb{R}$. If $S_f\leq 1/\sigma$, then
$\|f\ast S_\sigma-f\|_p=0$ for every $p\in[1,\,\infty]$.
If $S_f=\infty$, for every $p\in[2,\,\infty)$, repeat the same reasoning as for the
$\operatorname{sinc}$ kernel to conclude that, for every $\alpha\in(0,\,1)$,
$\|f\ast S_{\sigma}-f\|_p^p\leq C_p\|\hat{f}(\widehat{S_\sigma}-1)\|_q^q
=C_p\int_{|t|>1/\sigma}(|\hat{f}(t)||\hat{S}(\sigma t)-1|)^q\,\mathrm{d}t<
2^qC_p\int_{|t|>1/\sigma}|\hat{f}(t)|^q\,\mathrm{d}t\lesssim e^{-\alpha(\rho/\sigma)^{r}}$ because $|\hat{S}|<1$
outside $[-1,\,1]$.
The case where $p=\infty$ follows from the bound on $\int_{|t|>1/\sigma}|\hat{f}(t)|\,\mathrm{d}t$ in \eqref{infinity}.
Now, consider the case where $p\in[1,\,2)$.
From Lemma~1 in Devroye~\cite{Devroye92}, page~2040, and the assumption that $f\in\mathcal{A}^{\rho,\,r,\,L}(\mathbb{R})$,
$\|f\ast S_\sigma-f\|_1\leq 2(\int f^\upsilon\,\mathrm{d}\lambda)(\pi^{-1}\int_{|t|>1/\sigma}|\hat{f}(t)|\,\mathrm{d}t)^{1-\upsilon}\lesssim (\int f^\upsilon\,\mathrm{d}\lambda) e^{-\alpha(1-\upsilon)(\rho/\sigma)^{r}}$. For every $L^p$-metric, $p\in(1,\,2)$, use the inequality $\|f\ast S_\sigma-f\|_p\leq\max\{\|f\ast S_\sigma-f\|_1,\,\|f\ast S_\sigma-f\|_2\}$ (see, \emph{e.g.}, Athreya and Lahiri~\cite{AL}, page~104)
to conclude that $\|f\ast S_\sigma-f\|_p\lesssim e^{-\alpha(1-\upsilon)(\rho/\sigma)^{r}}$.
\end{proof}
Before proving Theorem~\ref{ThsuperKs},
a preliminary remark is in order. If $\hat{K}\in L^1(\mathbb{R})$,
then $\|\widehat{f_{F,\,\sigma}}\|_1\leq{(2\pi)}^{-1}\int|\hat{K}(\sigma t)|\,\mathrm{d}t<\infty$. If $K\in\mathcal{A}^{\rho,\,r,\,L}(\mathbb{R})$
for some constants $\rho,\,r,\,L>0$, then, not only is $\hat{K}\in L^1(\mathbb{R})$,
but $f_{F,\,\sigma}\in\mathcal{A}^{\rho\sigma,\,r,\,L/\sigma}(\mathbb{R})$.
The absolute integrability of $\hat{K}$ allows to recover
\emph{any} $f_{F,\,\sigma}$ by just inverting its Fourier transform.
\begin{proof}[Proof of Theorem~\ref{ThsuperKs}]
We appeal to Theorem~2 of Gin\'e and Nickl~\cite{GN11}, page~2891.
Choosing their $\gamma_n=1$ for all $n\in\mathds{N}$, we have
$\varepsilon_{n,p}:=\tilde{\varepsilon}_n(n\tilde{\varepsilon}_n^2)^{(1-1/p)/2}$,
where $\varepsilon_{n,p}$ and $\tilde{\varepsilon}_n$ play the same roles as $\delta_n$ and $\varepsilon_n$,
respectively, in the above cited theorem. If $\gamma=1$, fix $\psi\in[r^{-1},\,t]$.
For $s_n:=E(n\tilde{\varepsilon}_n^2)^{-1/\gamma}(\log n)^{\psi1_{\{1\}}(\gamma)}$, $E>0$ being a suitable constant,
let $\mathscr{P}_n:=\{f_{F,\,\sigma}:\,F\in\mathscr{M}(\mathbb{R}),\,\,\,\sigma\geq s_n\}$.
Note that, for every $f_{F,\,\sigma}\in\mathscr{P}_n$, we have
$I^{\rho_n,\,r}(f_{F,\,\sigma})\leq 2\pi L_n$,
with $\rho_n:=\rho s_n$ and $L_n:=L/s_n$.
Condition~1(a), \emph{ibidem}, page~2890, for the convolution kernel case
is verified for the $\operatorname{sinc}$ kernel. In fact,
$\operatorname{sinc}\in L^2(\mathbb{R})\cap L^\infty(\mathbb{R})$ since $\int\operatorname{sinc}^2\,\mathrm{d}\lambda=\|\operatorname{sinc}\|_\infty=1/\pi<\infty$.
Besides, the $\operatorname{sinc}$ kernel is continuous and, as shown in Lemma~\ref{BQV}, is of bounded quadratic variation. Let $\operatorname{sinc}_j(f):=f\ast\operatorname{sinc}_{2^{-j}}$, with the
usual conversion from bandwidth $\sigma$ to $2^{-j}$.
By Lemma~\ref{lem:admis-approx-seq}, for every
density $f_{F,\,\sigma}\in\mathscr{P}_n$ for which $S_{f_{F,\,\sigma}}<\infty$,
whatever sequence $J_n\rightarrow\infty$, for $n$ large enough so that $2^{J_n}>S_{f_{F,\,\sigma}}$,
we have $\|\operatorname{sinc}_{J_n}(f_{F,\,\sigma})-f_{F,\,\sigma}\|_p=0$ for every $p\in[2,\,\infty]$. Let
$\alpha\in(0,\,1)$ be fixed. For every density $f_{F,\,\sigma}\in\mathscr{P}_n$ for which $S_{f_{F,\,\sigma}}=\infty$, taking $J_n$ such that $2^{J_n}=cn\tilde{\varepsilon}_n^2$, with $c\geq \alpha^{-1/r}/(\rho E)$, and using the constraint on $\gamma$, we have $\|\operatorname{sinc}_{J_n}(f_{F,\,\sigma})-f_{F,\,\sigma}\|_p\lesssim\exp{(-\alpha(\rho s_n2^{J_n})^r)}\lesssim
\exp{(-\alpha(\rho E c)^r(n\tilde{\varepsilon}_n^2)^{r(1-1/\gamma)}(\log n)^{r\psi1_{\{1\}}(\gamma)})}\lesssim
n^{-1}\lesssim\varepsilon_{n,p}$ for every $p\in[2,\,\infty]$. Hence, for $n$ large enough, $\mathscr{P}_n\subseteq\{f_{F,\,\sigma}:\,\|\operatorname{sinc}_{J_n}(f_{F,\,\sigma})
-f_{F,\,\sigma}\|_p\leq C(K)\varepsilon_{n,p}\}$, where $C(K)>0$ is an appropriate constant depending only on the operator ($\operatorname{sinc}$) kernel. For $E\leq[(C+4)/(D_2-1_{(0,\,\infty)}(s))]^{-1/\gamma}$, where $C>0$ is the constant arising from the small ball probability estimate, the prior probability
of $\mathscr{P}^c_n$ is bounded above by
\[\begin{split}\operatorname{P}(\sigma\leq s_n)&\lesssim\exp{(-(D_2-1_{(0,\,\infty)}(s))s_n^{-\gamma}(\log n)^{t1_{\{1\}}(\gamma)})}\\&\lesssim\exp{(-(C+4)n\tilde{\varepsilon}_n^2(\log n)^{(t-\psi)1_{\{1\}}(\gamma)})}\lesssim
\exp{(-(C+4)n\tilde{\varepsilon}_n^2)}
\end{split}\]
for $n$ large enough and Assumption $(1)$, \emph{ibidem},
page~2891, is fulfilled.
\end{proof}
\begin{proof}[Proof of Corollary~\ref{Mixing}]
Under the stated conditions, Theorem~\ref{ThsuperKs} holds, with
$G$ a point mass at $1$, for $p=\infty$, because
$\|f_0\ast \operatorname{sinc}_{2^{-J_n}}-f_0\|_\infty=O(\varepsilon_{n,\infty})$,
with $\varepsilon_{n,\infty}:=\tilde{\varepsilon}_n(n\tilde{\varepsilon}_n^2)^{1/2}$. Thus,
there exists a sufficiently large constant $M>0$ so that
$\Pi(F:\,\|f_{F,\,1}-f_0\|_\infty<M\varepsilon_{n,\infty}|X^{(n)})\rightarrow1$ in $P_0^n$-probability.
Since $K$ is a symmetric density around $0$ such that, for some constants $\rho,\,r>0$,
$|\hat{K}(t)|\asymp e^{-(\rho t)^r}$ as $|t|\rightarrow\infty$,
by Theorem~2 of Nguyen~\cite{Ng?}, page~8, for any $F$ such that $\|f_{F,\,1}-f_0\|_\infty<M\varepsilon_{n,\infty}$,
we have $W_2(F,\,F_0)\lesssim(-\log\|f_{F,\,1}-f_0\|_1)^{-1/r}
\lesssim(\log n)^{-1/r}$, where the second inequality descends from Lemma~\ref{p<2} applied to $\|f_{F,\,1}-f_0\|_1$.
In fact, for some real $u>0$ such that $\operatorname{E}_K[|X|^u]<\infty$,
the absolute moment of order $u$ of $X$ under
$f_{F,\,1}$ is finite for \emph{every} $F\in\mathscr{M}(\Theta)$:
$\operatorname{E}_{f_{F,\,1}}[|X|^u]\leq (1\vee 2^{u-1})\{\operatorname{E}_K[|X|^u]+\int_{\Theta}|\theta|^u\,\mathrm{d}F(\theta)\}<\infty$,
the integral being finite because $F$ is compactly supported on $\Theta$.
Hence, for a suitable constant $M'>0$,
$\{F:\,\|f_{F,\,1}-f_0\|_\infty<M\varepsilon_{n,\infty}\}\subseteq\{F:\,W_2(F,\,F_0)<M'(\log n)^{-1/r}\}$
and the assertion follows.
\end{proof}

\subsection[Proof of Theorem~\ref{superKs}]{Proof of Theorem~\ref{superKs}}\label{ana-loc}
We preliminarily recall that if $(S,\,d)$ is a metric space and $C$ a
totally bounded subset of $S$, for any $\varepsilon>0$, the
$\varepsilon$-packing number of $C$, denoted $D(\varepsilon,\,C,\,d)$,
is defined as the largest integer $m$ such that there is a set
$\{s_1,\,\ldots,\,s_m\}\subseteq C$ with $d(s_k,\,s_l)>\varepsilon$ for
all $k,\,l=1,\,\ldots,\,m$, $k\neq l$. The $\varepsilon$-capacity of
$(C,\,d)$ is defined as $\log D(\varepsilon,\,C,\,d)$.

\begin{proof}[Proof of Theorem~\ref{superKs}]
We prove the result for the $L^1$-metric invoking Theorem~2.1
of Ghosal and van der Vaart~\cite{GvdV01}, page~1239. We deal
with $L^p$-metrics, $p\in[2,\,\infty]$, appealing to Theorem~\ref{ThsuperKs}.
For the cases where $p\in(1,\,2)$ the result follows from
$\|f_{F,\,\sigma}-f_0\|_p\leq
\max\{\|f_{F,\,\sigma}-f_0\|_1,\,\|f_{F,\,\sigma}-f_0\|_2\}\lesssim n^{-1/2}(\log n)^\varphi$ for
a suitable constant $\varphi>0$.

$\bullet$ \emph{$L^1$-metric}.
We show that conditions (2.8) and (2.9) in Theorem~2.1
of Ghosal and van der Vaart~\cite{GvdV01}, page~1239,
are satisfied for sequences $\bar{\varepsilon}_n=n^{-1/2}(\log n)^\chi$,
with a suitable constant $\chi>0$, and $\tilde{\varepsilon}_n=n^{-1/2}(\log n)^{\tau+(\tau-1/2)1_{(0,\,\infty)}(d)}$,
the latter arising from the small ball probability estimate below.
The posterior rate is $\varepsilon_{n,1}:=(\bar{\varepsilon}_n\vee\tilde{\varepsilon}_n)$.
Given $\eta_n\in(0,\,1/5)$, for constants $E,\,F,\,L>0$ to be
suitably chosen, let $s_n:=E(\log(1/\eta_n))^{-2[\tau+(\tau-1/2)1_{(0,\,\infty)}(d)]/\gamma}$, $S_n:=
\exp(F(\log(1/\eta_n))^{2[\tau+(\tau-1/2)1_{(0,\,\infty)}(d)]})$ and $a_n:=L(\log(1/\eta_n))^{2[\tau+(\tau-1/2)1_{(0,\,\infty)}(d)]/\delta}$.
For $\F_n:=\{f_{F,\,\sigma}:\,F([-a_n,\,a_n])\geq1-\eta_n,\,\,\,s_n\leq \sigma\leq S_n\}$,
by Lemma~A.3 of Ghosal and van der Vaart~\cite{GvdV01}, page~1261, and Lemma~\ref{L1-entropy},
\[\begin{split}
\log D(\eta_n,\,\F_n,\,\|\cdot\|_1)&\,\lesssim
\pt{\frac{a_n}{s_n}}^{1_{(0,\,1]}(r)}\,\times\,\pt{\log
\frac{1}{\eta_n}}^{1+1_{(0,\,1]}(r)/r}\\ &\qquad\qquad\qquad\quad\,\,\,\times\,\max\pg{\pt{\frac{a_n}{s_n}}^{r/(r-1)},\,\pt{\log \frac{1}{\eta_n}}}^{1_{(1,\,\infty)}(r)}.
\end{split}\] Taking
$\eta_n=\bar{\varepsilon}_n$, we have $\log
D(\bar{\varepsilon}_n,\,\F_n,\,\|\cdot\|_1)\lesssim
n\bar{\varepsilon}_n^2$. Regarding condition (2.9), by assumptions $(ii)$-$(iii)$
and the fact that $2\tau>1$, for appropriate choices of
$E,\,F,\,L$ as functions of the constant $c_2$
arising from the small ball probability estimate below,
the prior probability of $\F^c_n$ is bounded above by
$e^{-(D_2-1_{(0,\,\infty)}(s))s_n^{-\gamma}(\log( 1/s_n))^{t}}+S_n^{-\varrho}+e^{-ba_n^{\delta}}/\eta_n^2
\lesssim\exp{(-(c_2+4)n\tilde{\varepsilon}_n^2)}$
because, by Markov's inequality and the independence of $(W_j)_{j\geq1}$ and $(Z_j)_{j\geq1}$,
$\Pi(F:\, F([-a_n,\, a_n]^c)>\eta_n^2)<
\operatorname{E}[\sum_{j=1}^\infty W_j1_{[-a_n,\, a_n]^c}(Z_j)]/\eta_n^2
\lesssim\alpha([-a_n,\,a_n]^c)/\eta_n^2\lesssim e^{-ba_n^{\delta}}/\eta_n^2$.

$\bullet$ \emph{$L^p$-metrics, $p\in[2,\,\infty]$}. Conditions of
Theorem~\ref{ThsuperKs} are satisfied. Let $\varepsilon_{n,p}:=\tilde{\varepsilon}_n(n\tilde{\varepsilon}_n^2)^{(1-1/p)/2}$.
By the assumption that $f_0=f_{F_0,\,\sigma_0}=F_0\ast K_{\sigma_0}$, we have
$f_0\in\mathcal{A}^{\rho\sigma_0,\,r,\,L/\sigma_0}(\mathbb{R})$.
By Lemma~\ref{lem:admis-approx-seq}, for every $p\in[2,\,\infty]$,
letting $2^{J_n}=cn\tilde{\varepsilon}_n^2$ with $c$ defined as in the proof of Theorem~\ref{ThsuperKs},
$\|f_0\ast \operatorname{sinc}_{2^{-J_n}}-f_0\|_p=O(\varepsilon_{n,p})$ for $n$ large enough.

$\bullet$ \emph{Small ball probability estimate}.
We show that, for $0<\varepsilon\leq[(1/4)\wedge(\sigma_0/2)]$,
there exist constants $c_1,\,c_2>0$ so that
\[(\Pi\times G)(B_{\operatorname{KL}}(f_0;\,\varepsilon^2))\geq
c_1\exp(-c_2(\log(1/\varepsilon))^{2[\tau+(\tau-1/2)1_{(0,\,\infty)}(d)]}).\]
A preliminary remark is in order.
The case where $\varpi=\infty$ corresponds to $F_0$ having compact support, \emph{i.e.},
$F_0([-a_0,\,a_0])=1$ for some finite $a_0>0$.
Let $a_\varepsilon:=a_0^{1_{\{\infty\}}(\varpi)}(c_0^{-1}\log(1/\varepsilon))
^{1/\varpi}$ and let $F_0^*$ be the re-normalized restriction of $F_0$ to
$[-a_\varepsilon,\,a_\varepsilon]$. By Lemma~A.3
of Ghosal and van der Vaart~\cite{GvdV01}, page~1261, and
assumption $(\mathrm{A_2})$,
$\|{f_{F^*_0,\,\sigma_0}-f_0}\|_1\lesssim\varepsilon$.
We show that there exists a discrete probability measure $F_0'$
on $[-a_\varepsilon,\,a_\varepsilon]$, with at most
\begin{equation}\label{NTpoints}
N\lesssim\pt{\log\frac{1}{\varepsilon}}^{2\tau-1}
\end{equation}
support points, such that $\|f_{F_0^*,\,\sigma_0}-f_{F_0',\,\sigma_0}\|_\infty\lesssim\varepsilon$.
The support points of $F_0'$ can be taken to be at least
$2\varepsilon$-separated. We distinguish the case
where $r\in(0,\,1]$ from the case where $r>1$. In the latter case, the assertion follows
immediately from Lemma~\ref{Npoints}: in fact, $a_\varepsilon$ can be taken to be large enough
so that $a_\varepsilon/(\rho\sigma_0)\geq e^{-1}$.
If $r\in(0,\,1]$, Lemma~\ref{Npoints} cannot be directly applied
because the requirement on $a_\varepsilon/(\rho\sigma_0)$
may not be met. Yet, an argument similar to the one used
in Lemma~2 of Ghosal and van der Vaart~\cite{GvdV072}, page~705, can be adopted.
Consider a partition of $[-a_\varepsilon,\,a_\varepsilon]$ into
$k=\lceil{a_0^{1_{\{\infty\}}(\varpi)}(c_0^{(1-1_{\{\infty\}}(\varpi))/\varpi}\sigma_0)^{-1}
(\log(1/\varepsilon))^{1/r-1+1_{(0,\,\infty)}(\varpi)/\varpi}}\rceil$
subintervals $I_1,\,\ldots,\,I_k$
of equal length $0<l\leq2\sigma_0(\log(1/\varepsilon))^{-(1-r)/r}$
and, possibly, a final interval $I_{k+1}$
of length $0\leq l_{k+1}<l$. Let $J$ be the total number of intervals in the partition,
which can be either $k$ or $k+1$. Write $F_0^*=\sum_{j=1}^{J}F_0^*(I_j)F_{0,j}^*$, where
$F_{0,j}^*$ denotes the re-normalized restriction of $F_0^*$ to $I_j$. Then,
$f_{F_0^*,\,\sigma_0}(x)=\sum_{j=1}^{J}F_0^*(I_j)f_{F_{0,j}^*,\,\sigma_0}(x)
=\sum_{j=1}^{J}F_0^*(I_j)(F_{0,j}^*\ast K_{\sigma_0})(x)$, $x\in\mathbb{R}$.
For every $j=1,\,\ldots,\,J$, by Lemma~\ref{Npoints} (and Remark~\ref{compactness}) applied to every $f_{F_{0,j}^*,\,\sigma_0}$, with $a/\sigma=(l/2)/\sigma_0\propto(\log(1/\varepsilon))^{-(1-r)/r}$, there exists
a discrete distribution $F_{0,j}'$, with at most $N_j\lesssim\log(1/\varepsilon)$
support points, such that
$\|f_{F_{0,j}^*,\,\sigma_0}-f_{F_{0,j}',\,\sigma_0}\|_\infty\lesssim\varepsilon$.
Defined $F_0':=\sum_{j=1}^{J}F_0^*(I_j)F_{0,j}'$, we have
$\|f_{F_0^*,\,\sigma_0}-f_{F_0',\,\sigma_0}\|_\infty\leq
\sum_{j=1}^{J}F_0^*(I_j)\|f_{F_{0,j}^*,\,\sigma_0}-f_{F_{0,j}',\,\sigma_0}\|_\infty
\lesssim\varepsilon$,
where $F_0'$ has at most
$N\lesssim \sum_{j=1}^JN_j\lesssim k\times\log(1/\varepsilon)\lesssim
(\log(1/\varepsilon))^{1/r+1_{(0,\,\infty)}(\varpi)/\varpi}$ support points.
Combining the result on the total number $N$ of support points of $F_0'$
in the case where $r\in(0,\,1]$ with the one in the case where $r>1$, we obtain the bound
in \eqref{NTpoints}. Let $q>0$ real be such that $\operatorname{E}_K[|X|^q]<\infty$.
For any $\upsilon$ such that $(1+q)^{-1}<\upsilon<1$, by Hölder's inequality, $\int f^\upsilon_{F^*_0,\,\sigma_0}\,\mathrm{d}\lambda\lesssim(1+\int|x|^q
f_{F^*_0,\,\sigma_0}(x)\,\mathrm{d}x)^\upsilon\lesssim \{(1\vee 2^{q-1})[\sigma_0^q\operatorname{E}_K[|X|^q]+\int_{-a_\varepsilon}^{a_\varepsilon}|\theta|^q\,\mathrm{d}F^*_0(\theta)]\}^\upsilon
\lesssim a_\varepsilon^{\upsilon q}$, this implying that $\|{f_{F^*_0,\,\sigma_0}-f_{F_0',\,\sigma_0}}\|_1\lesssim
\varepsilon^{1-\upsilon}a_\varepsilon^{\upsilon q}$ in virtue of Lemma~\ref{L1tosup}.

Next, we distinguish the case where the prior for $F$ is
a Dirichlet process, \emph{i.e.}, a Pitman-Yor process
with $d=0$ and $c=\alpha(\mathbb{R})$, from the case where the prior for $F$
is a general Pitman-Yor process with $d\in[0,\,1)$ and $c>-d$.
The proof for the Dirichlet process is paradigmatic to deal with
other process priors, like the N-IG process, whose finite-dimensional distributions are known.

$-$ \emph{Dirichlet process}. Represented $F'_0$ as $\sum_{j=1}^Np_j\delta_{\theta_j}$,
with $|\theta_j-\theta_k|\geq2\varepsilon$ for all $j\neq k$,
and set $U_j:=[\theta_j-\varepsilon,\,\theta_j+\varepsilon]$, $j=1,\,\ldots,\,N$,
for every $F\in\mathscr{M}(\mathbb{R})$ such that
\begin{equation}\label{eqIV}
\sum_{j=1}^N
|{F(U_j)-p_j}|\leq\varepsilon,
\end{equation}
and every $\sigma>0$ such that $|{\sigma-\sigma_0}|\leq\varepsilon$,
we have $\|{f_{F,\,\sigma}-f_{F'_0,\,\sigma_0}}\|_1\lesssim
\|K_{\sigma}-K_{\sigma_0}\|_1
+\varepsilon/(\sigma\wedge\sigma_0)+
\sum_{j=1}^N|{F(U_j)-p_j}|\lesssim
\varepsilon$ in virtue of Lemma~\ref{L1norm}, Lemma~\ref{Ksigma} and
condition \eqref{eqIV}. Thus, $\|{f_{F,\,\sigma}-f_{F'_0,\,\sigma_0}}\|_1\lesssim\varepsilon$
and $h^2(f_{F,\,\sigma},\,f_0)\leq
\|{f_{F,\,\sigma}-f_{F'_0,\,\sigma_0}}\|_1+\|{f_{F'_0,\,\sigma_0}-f_{F_0^*,\,\sigma_0}}\|_1+
\|{f_{F^*_0,\,\sigma_0}-f_0}\|_1\lesssim\varepsilon^{1-\upsilon}a_\varepsilon^{\upsilon q}$.
In order to appeal to Theorem~5 of Wong and Shen~\cite{WS}, pages~357--358,
we show that, for densities in the set
$S_\varepsilon:=\{f_{F,\,\sigma}:\,\sum_{j=1}^N
|{F(U_j)-p_j}|\leq\varepsilon,
\,\,\,|{\sigma-\sigma_0}|\leq\varepsilon\}$ and a suitable constant $\varrho\in(0,\,1]$, we have
$M_\varrho^2:=\int_{\{(f_0/f_{F,\,\sigma})\geq e^{1/\varrho}\}} f_0(f_0/f_{F,\,\sigma})^\varrho\,\mathrm{d}\lambda
=O((1/\varepsilon)^{\xi})$, with $0\leq\xi\leq \kappa/\varpi$.
For every $F$ satisfying (\ref{eqIV}),
$F([-a_\varepsilon,\,a_\varepsilon])>\frac{1}{2}$, thus, by symmetry and monotonicity of $K$,
$f_{F,\,\sigma}(x)\geq
\int_{-a_\varepsilon}^{a_\varepsilon}K_\sigma(x-\theta)\,\mathrm{d}F(\theta)>\frac{1}{2}
K_\sigma(|x|+a_\varepsilon)$, $x\in\mathbb{R}$.
By assumption $(\mathrm{A_1})$, $K(a_\varepsilon)\gtrsim\exp{(-ca_\varepsilon^\kappa)}$
for $a_\varepsilon$ large enough so that
$\int_{|x|\leq a_\varepsilon}f_0^{1+\varrho}(x)K_\sigma^{-\varrho}(|x|+a_\varepsilon)\,\mathrm{d}x
\lesssim\exp{(\varrho c(4a_\varepsilon/\sigma_0)^\kappa)}$
because $|\sigma-\sigma_0|\leq\varepsilon\leq\sigma_0/2$ and $\|f_0\|_\infty<\infty$.
Also,
\[\int_{|x|>a_\varepsilon}\frac{f_0^{1+\varrho}(x)}
{K_\sigma^\varrho(|x|+a_\varepsilon)}\,\mathrm{d}x\lesssim\int_{|x|>a_\varepsilon}K_{\sigma_0}^{-\varrho}(4|x|)[
K_{\sigma_0}(|x|/2)+F_0(\theta:\,|\theta|>|x|/2)]\,\mathrm{d}x<\infty,
\]
where the last integral is finite for a suitable choice of $\varrho$ and in virtue of
assumption $(\mathrm{A_2})$. Thus, $S_\varepsilon\subseteq\,
B_{\operatorname{KL}}(f_0;\,c_1\varepsilon^{1-\upsilon}a_\varepsilon^{\upsilon q}(\log(1/\varepsilon))^2)$.
To apply Lemma~A.2 of Ghosal and van der Vaart~\cite{GvdV01}, pages~1260--1261,
note that, for each $|\theta_j|\leq a_\varepsilon$, by assumption $(\mathrm{A_3})$,
$\alpha(U_j)\gtrsim
\varepsilon e^{-ba_{\varepsilon}^\delta}\gtrsim\varepsilon^{b'}$
for some constant $b'>0$ because, when $\varpi<\infty$, we have $\delta\in(0,\,\varpi]$ by assumption.
Thus, 
$\tilde{\varepsilon}_n=n^{-1/2}(\log n)^\tau$.

$-$ \emph{Pitman-Yor process with $d\in[0,\,1)$ and $c>-d$}. We need to modify the arguments to control
$\|{f_{F,\,\sigma}-f_{F'_0,\,\sigma_0}}\|_1$. To the aim, the stick-breaking representation of
$F$ is exploited. Let $F'_0=\sum_{j=1}^Np_j\delta_{\theta_j}$
be the finite approximating distribution of $F_0^\ast$. By relabelling, we can
assume that $p_1\geq p_2\geq\,\ldots\geq\,p_N\geq0$.
Let $1\leq M\leq N$ be the number of strictly positive mixing weights.
For every $\sigma>0$, by Lemma~\ref{L1norm} and the inequality $\sum_{j=M+1}^\infty W_j\leq\sum_{j=1}^M|W_j-p_j|$,
\begin{equation}\label{SBd}
\|f_{F,\,\sigma}-f_{F_0',\,\sigma}\|_1\leq
2\sum_{j=1}^M|W_j-p_j|+\frac{2\|K\|_\infty}{\sigma}\sum_{j=1}^Mp_j|Z_j-\theta_j|.
\end{equation}
Let $v_1:=p_1$ and
$v_j:=p_j[\prod_{h=1}^{j-1}(1-v_h)]^{-1}$ for $j=2,\,\ldots,\,M$. Note that
$v_j\in(0,\,1)$ for every $j=1,\,\ldots,\,M$.
We have $|W_j-p_j|\leq|V_j-v_j|\prod_{h=1}^{j-1}(1-V_h)+v_j|\prod_{h=1}^{j-1}(1-V_h)-\prod_{h=1}^{j-1}(1-v_h)|
\leq\sum_{h=1}^j|V_h-v_h|$,
where the inequality
$|\prod_{h=1}^{j-1}y_h-\prod_{h=1}^{j-1}z_h|\leq\sum_{h=1}^{j-1}|y_h-z_h|$, valid for complex numbers
$y_1,\,\ldots,\,y_{j-1}$ and $z_1,\,\ldots,\,z_{j-1}$ of modulus at most $1$, has been used.
If, for $0<\varepsilon\leq\sigma_0/2$,\\[10pt]
$a)\,\,\sum_{j=1}^M\sum_{h=1}^j|V_h-v_h|\leq\varepsilon$, \hfill $b)\,\,\sum_{j=1}^M|Z_j-\theta_j|\leq\varepsilon$,\hfill $c)\,\,|{\sigma-\sigma_0}|\leq\varepsilon$,\\[10pt]
then $\|{f_{F,\,\sigma}-f_{F'_0,\,\sigma_0}}\|_1\lesssim
\|K_{\sigma}-K_{\sigma_0}\|_1
+\sum_{j=1}^M\sum_{h=1}^j|V_h-v_h|
+\sum_{j=1}^Mp_j|Z_j-\theta_j|
\lesssim\varepsilon$ by Lemma~\ref{Ksigma} and inequality~\eqref{SBd}.
Next, we show that, for $B_\varepsilon=a_\varepsilon$
(or $B_\varepsilon=a_\varepsilon+1$, the latter case being considered if
any support point $\theta_j$ of $F_0'$ is equal to $-a_\varepsilon$ and/or $a_\varepsilon$),
the events in $a)$ and $b)$ together imply that, for $0<\varepsilon\leq[(1/4)\wedge(\sigma_0/2)]$, we have
$F([-B_\varepsilon,\,B_\varepsilon])>\frac{1}{2}$. This inequality is used when checking that,
for a suitable $\varrho\in(0,\,1]$,
$M_\varrho^2=O((1/\varepsilon)^{\xi})$, with $0\leq\xi\leq \kappa/\varpi$,
so that Theorem~5 of Wong and Shen~\cite{WS}, pages~357--358,
can be invoked. By the event in $b)$, for $\varepsilon>0$ small enough, all the $Z_j$'s are in $[-B_\varepsilon,\,B_\varepsilon]$.
Using this fact and the inequality
$\sum_{j=1}^M|W_j-p_j|\leq\sum_{j=1}^M\sum_{h=1}^j|V_h-v_h|$,
the event in $a)$ implies that
$F([-B_\varepsilon,\,B_\varepsilon]^c)\leq\sum_{j=1}^M\sum_{h=1}^j|V_h-v_h|\leq\varepsilon<\frac{1}{2}$.

Next, we estimate the probabilities of the events in $a)$ and $b)$. By the independence of
$(W_j)_{j\geq1}$ and $(Z_j)_{j\geq1}$, Lemma~\ref{PDP1} and Lemma~\ref{PDP2}, when $d>0$, for $(1-v_{\max})>4\varepsilon/M^2$ (if $v_{\max}$ does not satisfy the condition, $f_{F_0',\,\sigma_0}$ can be projected into a new density $f_{F_0'',\,\sigma_0}$ which is within $\varepsilon$ $L^1$-distance from $f_{F_0',\,\sigma_0}$: this new density can be obtained by first changing the point mass $p_{\mathrm{m}}$ corresponding to $v_{\mathrm{max}}$
into some $p'_{\mathrm{m}}$ such that $(1-v'_{\max})>4\varepsilon/M^2$ and then distributing the remaining mass among the other $M-1$ points so that $v'_{\max}$ is still the maximum), we have
\begin{eqnarray*}
\operatorname{P}\pt{
\sum_{j=1}^M|W_j-p_j|\leq\varepsilon}\times
\operatorname{P}\pt{
\sum_{j=1}^M|Z_j-\theta_j|\leq\varepsilon}\gtrsim\exp{(-c_2M^2\log(1/\varepsilon))},
\end{eqnarray*}
because, by \eqref{NTpoints}, $1\leq M\leq N\lesssim(\log(1/\varepsilon))^{2\tau-1}$, where $\tau\geq1$,
and, for $\varpi<\infty$, we have $\delta\in(0,\,\varpi]$ by assumption, so that
$a_\varepsilon^\delta\lesssim \log(1/\varepsilon)$. Thus,
$\tilde{\varepsilon}_n=n^{-1/2}(\log n)^{2\tau-1/2}$. For $d=0$,
the same lower bound as for the Dirichlet process is obtained.
\end{proof}


\subsection[Proof of Theorem~\ref{adaptiveSobolev}]{Proof of Theorem~\ref{adaptiveSobolev}}\label{ada-sobo}
Before proving the theorem, we present some auxiliary results.
For any real $\sigma>0$ and function $f_0$ having derivatives up to the
order $k_0\in\mathbb{N}$, we define the transform
\[T_{k_0,\,\sigma}(f_0):=f_0-\sum_{j=1}^{k_0-1} d_j\sigma^j(f_0^{(j)}\ast\operatorname{sinc}_\sigma),\]
where the $d_j$'s are as defined in \eqref{d_j's}. The following approximation result holds.

\begin{lem}\label{approxlem2}
Let $f_0\in W^{k_0,\,2}(\mathbb{R})$, $k_0\in\mathbb{N}$, be a probability density.
For $\sigma>0$ small enough, $\|T_{k_0,\,\sigma}(f_0)\ast\phi_\sigma-f_0\|_\infty
\lesssim \sigma^{k_0-1/2}$.
\end{lem}
\begin{proof}
By definition of $T_{k_0,\,\sigma}(f_0)$ and of the $d_j$'s, using Taylor's theorem with the
integral form of the remainder, for every $x\in\mathbb{R}$,
\[\begin{split}
&\hspace*{-0.1cm}(T_{k_0,\,\sigma}(f_0)\ast\phi_\sigma-f_0)(x)\\
&\hspace*{0.5cm}=\sum_{j=1}^{k_0-1}\pt{\frac{(-1)^jm_j}{j!}\sigma^j(f_0^{(j)}-f_0^{(j)}\ast \operatorname{sinc}_\sigma\ast\phi_\sigma)(x)-c_j\sigma^j(f_0^{(j)}\ast\operatorname{sinc}_\sigma\ast\phi_\sigma)(x)}\\
&\hspace*{9.6cm}+\int R_{k_0}(x,\,y)\phi_\sigma(y)\,\mathrm{d}y\\
&\hspace*{0.5cm}\lesssim
\sum_{j=1}^{k_0-1}\pq{\frac{(-1)^jm_j}{j!}\sigma^jT_2(j,\,\sigma,\,x)
-c_j\sigma^j(f_0^{(j)}\ast\operatorname{sinc}_\sigma\ast\phi_\sigma)(x)}+\sigma^{k_0-1/2}\lesssim \sigma^{k_0-1/2},
\end{split}
\]
where
\begin{equation}\label{remainder}
R_{k_0}(x,\,y):=\frac{(-y)^{k_0}}{(k_0-1)!}\int_0^1(1-s)^{k_0-1}f_0^{(k_0)}(x-sy)\,\mathrm{d}s
\end{equation}and
$T_2(j,\,\sigma,\,x):=(2\pi)^{-1}
\int(-it)^je^{-itx}\widehat{f_0}(t)1_{[-1,\,1]}(\sigma t)\,\mathrm{d}t-
(f_0^{(j)}\ast\operatorname{sinc}_\sigma\ast\phi_\sigma)(x)$.
The following facts have been used. By the Cauchy-Schwarz inequality and
the assumption that $f_0\in W^{k_0,\,2}(\mathbb{R})$, $\|f_0^{(j)}-f_0^{(j)}\ast\operatorname{sinc}_\sigma\|_\infty\lesssim
{(2\pi)^{-1}}\,
\int_{|t|>1/\sigma}|t|^j|\widehat{f_0}(t)|\,\mathrm{d}t\lesssim
\sigma^{-j+k_0-1/2}$ for every $j=1,\,\ldots,\,k_0-1$.
Also, since $\sup_{x\in\mathbb{R}}|R_{k_0}(x,\,y)|\lesssim|y|^{k_0-1/2}$ and $f_0^{(k_0)}\in L^2(\mathbb{R})$,
$\int|R_{k_0}(x,\,y)|\phi_\sigma(y)\,\mathrm{d}y\lesssim\sigma^{k_0-1/2}$.
To conclude, note that the sum in the last display 
is identically equal to zero.
\end{proof}
\begin{rmk}\label{integral=1}
\emph{Define $T_{k_0,\,\sigma}(f_0):=f_0-\sum_{j=1}^{k_0-1}d_j\sigma^j(f_0^{(j)}\ast S_\sigma)$,
where $S$ is a superkernel. If $(f_0^{(j)}\ast S_\sigma)\in L^1(\mathbb{R})$ for every $j=1,\,\ldots,\,k_0-1$, then $\int T_{k_0,\,\sigma}(f_0)\,\mathrm{d}\lambda=1$.
The integrability conditions in assumption $(a')$ imply that $(f_0^{(j)}\ast S_\sigma)\in L^1(\mathbb{R})$ for every $j=1,\,\ldots,\,k_0-1$.}
\end{rmk}

Suppose that $f_0$ satisfies condition $(a')$ for some $k_0\in\mathbb{N}$.
Let $\delta:=(2-\sqrt{e})$. For given reals
$B,\,\sigma>0$ and $M:=4(k_0+\frac{1}{2})$, let
\[
\begin{split}
B_\sigma&:=\{x\in\mathbb{R}:\,f_0(x)\geq B\sigma^{M}\},\\
G_{\sigma}&:=\{x\in\mathbb{R}:\,T_{k_0,\,\sigma}(f_0)(x)>\delta f_0(x)\},\\
U_\sigma&:=\{x\in\mathbb{R}:\,|(f_0^{(j)}\ast S_\sigma)(x)|\leq \sigma^{-j}f_0(x)/\sqrt{e}
,\,\,\,j\in\mathbb{N}\}.
\end{split}
\]

\begin{lem}\label{lem:intg2}
Suppose that $f_0$ satisfies condition $(a')$ for some $k_0\in\mathbb{N}$.
Let $g_{k_0,\,\sigma}:=T_{k_0,\,\sigma}(f_0)1_{G_{\sigma}}+\delta f_01_{G^c_{\sigma}}$,
with $T_{k_0,\,\sigma}(f_0)$ defined as in Remark~\ref{integral=1}
with a superkernel $S$.
For $\sigma>0$ small enough, $\delta\leq\int g_{k_0,\,\sigma}\,\mathrm{d}\lambda=1+O(\sigma^{2k_0-1})$.
\end{lem}
\begin{proof}
By definition, $g_{k_0,\,\sigma}\geq\delta f_0$ so that
$\int g_{k_0,\,\sigma}\,\mathrm{d}\lambda\geq\delta$.
Write $g_{k_0,\,\sigma}=T_{k_0,\,\sigma}(f_0)+\pq{\delta f_0-T_{k_0,\,\sigma}(f_0)}1_{G_\sigma^c}$.
By the integrability conditions in $(a')$ and Remark~\ref{integral=1},
$\int g_{k_0,\,\sigma}\,\mathrm{d}\lambda=1+
\int\pq{\delta f_0-T_{k_0,\,\sigma}(f_0)}1_{G_\sigma^c}\,\mathrm{d}\lambda$.
We prove that $\int\pq{\delta f_0-T_{k_0,\,\sigma}(f_0)}1_{G_\sigma^c}\,\mathrm{d}\lambda=O(\sigma^{2k_0-1})$.
We begin to show that $U_\sigma\subseteq G_\sigma$. Since $\sum_{j=1}^\infty |d_j|\leq (\sqrt{e}-1)\sqrt{e}$,
over the set $U_\sigma$, $|T_{k_0,\,\sigma}(f_0)-f_0|\leq f_0e^{-1/2}
\sum_{j=1}^{k_0-1}|d_j|\leq(\sqrt{e}-1)f_0$. Hence,
$T_{k_0,\,\sigma}(f_0)>\delta f_0$ and $U_\sigma\subseteq G_\sigma$.
The set $U_\sigma^c$ has exponentially small probability.
By Markov's inequality and the integrability conditions in $(a')$,
$\operatorname{P}_0(U_\sigma^c)\lesssim
\sigma^{2k_0-1}\sum_{j=1}^{k_0-1}
\operatorname{E}_0[|((f_0^{(j)}\ast S_\sigma)/f_0)(X)|^{(2k_0-1)/j}]\lesssim\sigma^{2k_0-1}$.
It follows that
$\int[\delta f_0-T_{k_0,\,\sigma}(f_0)]1_{U_\sigma^c}\,\mathrm{d}\lambda\lesssim \operatorname{P}_0(U_\sigma^c)\lesssim\sigma^{2k_0-1}$.
\end{proof}

The following lemma can be proved similarly to Theorem~2 in Maugis and Michel~\cite{MM11}.

\begin{lem}\label{appro2}
Suppose that $f_0$ satisfies conditions $(a')$ for some $k_0\in\mathbb{N}$, $(b)$ and $(c)$.
Assume that, for any $\sigma>0$, $\operatorname{E}_0[((1+\int R_{k_0}(X,\,\sigma z)\phi(z)\,\mathrm{d}z)/f_0(X))^2]<\infty$, with $R_{k_0}(\cdot,\,\cdot)$ as in \eqref{remainder}.
Then, for $\sigma>0$ small enough, there exists a finite Gaussian
mixture $m_\sigma$, having at most $N_\sigma=O(a_\sigma/\sigma)$ support points
in $[-a_\sigma,\,a_\sigma]$, with $a_\sigma=O((\log(1/\sigma))^{1/2})$, such that
$\max\{\operatorname{KL}(f_0;\,m_\sigma),\,\operatorname{E}_0[(\log(f_0/m_\sigma))^2]\}\lesssim
\sigma^{2k_0-1}$.
\end{lem}

\begin{proof}[Proof of Theorem~\ref{adaptiveSobolev}]
We prove the result for the $L^1$-metric and the $L^2$-metric.
The case of $L^p$-metrics, $p\in(1,\,2)$, is covered by interpolation.

$\bullet$ \emph{$L^1$-metric}. The entropy condition (2.8) and the small ball probability estimate condition
(2.10) of Theorem~2.1 of Ghosal and van der Vaart~\cite{GvdV01}, page~1239,
are shown to be satisfied for
$\bar{\varepsilon}_n=n^{-(1-1/2k_0)/2}(\log n)^{\tau+5/4}$
and $\tilde{\varepsilon}_n=n^{-(1-1/2k_0)/2}
(\log n)^{\tau}$, respectively, with $\tau$ as in (\ref{Ka}).
The posterior rate is
$\varepsilon_{n,1}:=(\bar{\varepsilon}_n\vee\tilde{\varepsilon}_n)=\bar{\varepsilon}_n$.
We start by considering the entropy condition.
For $a,\,s>0$ and $0<\eta<1$, let
$\F_{a,\,\eta,\,s,\,S}:=\{f_{F,\,\sigma}:\, F([-a,\,
a])\geq1-\eta,\,\,\, s\leq\sigma\leq S\}$
and
$\F_{a,\,s,\,S}:=\{f_{F,\,\sigma}:\, F([-a,\, a])=1,\,\,\,
s\leq\sigma\leq S\}$.
Combining Lemma~A.3 in Ghosal and van der Vaart~\cite{GvdV01},
page~1261, with Lemma~3 in Ghosal and van der Vaart~\cite{GvdV072},
pages~705--707,
\[\begin{split}
\log
D(\eta,\,\F_{a,\,\eta/2,\,s,\,S},\,\norm{\cdot}_1)&\leq
\log N(3\eta/2,\,\F_{a,\,s,\,S},\,\norm{\cdot}_1)\\
&\lesssim
\log\pt{\frac{2S}{3\eta s}}\\&\quad\qquad+\pt{\frac{a}{s}\vee 1}\pt{\log\frac{2}{3\eta}}\times
\pq{\log\pt{\frac{2a}{3\eta s}+1}+\log\frac{2}{3\eta}}.
\end{split}\]
Choosing $\eta_n=\bar{\varepsilon}_n$, $s_n=E(n\tilde{\varepsilon}_n^2)^{-1}$
and $a_n=L(\log n)^{1/2}$ with suitable constants $E,\,L>0$,
for $\F_n:=\F_{a_n,\,\eta_n/2,\,s_n,\,S}$, we have $\log D(\bar{\varepsilon}_n,\,\F_n,\,\norm{\cdot}_1)\lesssim n\bar{\varepsilon}^2_n$.

Next, we present the proof of the remaining mass condition
in the case where the prior for $F$ is a Dirichlet process.
The result in the case where the prior for $F$ is
a N-IG process only requires suitable modifications of
the arguments in Lemma~11 of Ghosal and van der Vaart~\cite{GvdV072},
pages~715--717.

$-$ \emph{Dirichlet process}. The posterior probability of $\F_n^c$
is bounded above by $\operatorname{P}(\sigma<s_n|X^{(n)})+
\operatorname{P}(F([-a_n,\,a_n]^c)>\eta_n/2|X^{(n)})
=:T^{(1)}_n+T^{(2)}_n$.
The term $T^{(1)}_n\overset{\operatorname{P}}{\rightarrow}0$ because,
if, as shown below, the small ball probability estimate condition
$(\Pi\times G)(B_{\mathrm{KL}}(f_0;\,\tilde{\varepsilon}_n^2))\gtrsim
\exp{(-c_2n\tilde{\varepsilon}_n^2)}$
is satisfied, in virtue of Lemma~1 in Ghosal and
van der Vaart~\cite{GvdV071}, page~195, (see also Lemma~5 of Barron
\emph{et al.}~\cite{BSW}, pages~543--544), it is enough that, for some constant $c>0$,
$\operatorname{P}(\sigma<s_n)\lesssim \exp{(-cn\tilde{\varepsilon}_n^2)}$, which
holds true for the above choice of $s_n$. We now show that $\operatorname{E}_0^n[T^{(2)}_n]\rightarrow0$.
For $n$ large enough so that $a_n\geq1$, by Lemma~11 of Ghosal and van der Vaart~\cite{GvdV072},
pages~715--717,
\[\begin{split}
\operatorname{E}_0^n[T^{(2)}_n]&\leq
\operatorname{E}_0^n[\Pi(F([-a_n,\,a_n]^c)>\eta_n/2|X^{(n)}
)\,1_{\{\max_{1\leq
i\leq n}|X_i|\leq a_n/2\}}]\nonumber\\&\hspace*{7.5cm}+\operatorname{E}_0^n[1_{\{\max_{1\leq i\leq
n}|X_i|> a_n/2\}}]\nonumber\\
&\lesssim
\frac{\alpha([-a_n,\,a_n]^c)}{\eta_n[\alpha(\mathbb{R})+n]}+
\frac{n\exp{(-a_n^2/(16S^2))}}{\eta_n
\lambda_n}+\operatorname{E}_0^n[1_{\{\max_{1\leq i\leq
n}|X_i|>a_n/2\}}],
\end{split}
\]
where $\lambda_n:=\inf_{|\theta|\leq a_n}\alpha'(\theta)>0$. Note that
$\eta_n^{-1}[\alpha(\mathbb{R})+n]^{-1}\alpha([-a_n,\,a_n]^c)\lesssim (n\eta_n)^{-1}\rightarrow0$.
Recalling that $\delta\in(0,\,2]$, $S\geq1$ and $a_n=L(\log n)^{1/2}$,
taking $L>\max\{4(3+b)^{1/2}S,\,2(2/c_0)^{1/\varpi}\}$, for $n$ large enough,
\[\frac{n\exp{(-a_n^2/(16S^2))}}{\eta_n
\lambda_n}\leq n \exp{\pg{-\pq{\frac{a_n^2}{16S^2}-(\log n)-ba_n^\delta}}}
<\frac{1}{n}.\]
Under assumption $(c)$ on the tails of $f_0$,
$nP_0\pt{|X_1|>a_n/2}\lesssim ne^{-c_0(a_n/2)^\varpi/2}\rightarrow0$,
this implying that $\mathbb{E}_0^n[1_{\{\max_{1\leq i\leq n}|X_i|>
a_n/2\}}]\rightarrow0$.

$\bullet$ \emph{$L^2$-metric}. We appeal to Theorem~3 in Gin\'e and Nickl~\cite{GN11}, page~2892.
Choosing their $\gamma_n=1$ for all $n\in\mathds{N}$, we have
$\varepsilon_{n,2}:=\tilde{\varepsilon}_n$. Condition $(b)$ that
$\tilde{\varepsilon}^2_n=O(n^{-1/2})$ is verified for every $k_0\in\mathbb{N}$.
Condition $(1)$ can be shown to be verified as in the proof of Theorem~\ref{ThsuperKs}.
By the assumption that $f_0\in W^{k_0,\,2}$, $k_0\in\mathbb{N}$, we have
$f_0\in L^\infty(\mathbb{R})$ and,
taking $2^{J_n}=cn\tilde{\varepsilon}_n^2$, with $c$ defined as in the proof of Theorem~\ref{ThsuperKs},
$\|f_0\ast \operatorname{sinc}_{2^{-J_n}}-f_0\|_2=O(\varepsilon_{n,2})$.
Concerning condition $(3)$, we first apply Theorem~2, \emph{ibidem}, page~2891, for the sup-norm
(note that the condition $\|f_0\ast \operatorname{sinc}_{2^{-J_n}}-f_0\|_\infty=O(n^{1/2}\tilde{\varepsilon}_n^2)$ is satisfied)
and then use the conclusion that the posterior concentrates on a shrinking sup-norm neighborhood of $f_0$
to see that the posterior accumulates on a fixed sup-norm ball of radius $B:=1+\|f_0\|_\infty$
with probability tending to one.


$\bullet$ \emph{Small ball probability estimate}.
By routine computations, it can be seen that, for the Dirichlet and the N-IG process,
there exists a constant $c_2>0$ so that
$(\Pi\times G)(B_{\mathrm{KL}}(f_0;\,\tilde{\varepsilon}_n^2))\gtrsim
\exp{(-c_2n\tilde{\varepsilon}_n^2)}$ for $\tilde{\varepsilon}_n=n^{-(1-1/2k_0)/2}
(\log n)^{\tau}$, with $\tau$ as in (\ref{Ka}).
\end{proof}

\subsection[Auxiliary results]{Auxiliary results}\label{AppendixC}
This subsection reports some auxiliary results used throughout the article.
Proofs that are an adaptation of those of results known in the literature are omitted.
In the following lemma, the $\operatorname{sinc}$ kernel is shown to have bounded quadratic variation.
By definition, a function $h$ is of bounded $p$-variation on $\mathbb{R}$, $p\geq1$ real, if
$v_p(h):= \sup\{\pt{\sum_{k=1}^n|h(x_k)-h(x_{k-1})|^p}^{1/p}:\,-\infty<x_0<\,\ldots\,<x_n<\infty,\,\,n\in\mathbb{N}\}$
is finite.
\begin{lem}\label{BQV}
The function $x\mapsto\operatorname{sinc}(x)$ has bounded quadratic variation.
\end{lem}
\begin{proof}
It is shown that $v_2(\operatorname{sinc})<\infty$. For every $n\in\mathbb{N}$,
$\sum_{k=1}^n[\operatorname{sinc}(x_k)-\operatorname{sinc}(x_{k-1})]^2$ is maximum at
$x_k:=(2k+1)\pi/2$, $k=1,\,\ldots,\,n$. Splitting the sum into two parts,
\[
\sum_{1\leq k=2j\leq n}
[\operatorname{sinc}(x_k)-\operatorname{sinc}(x_{k-1})]^2=
\frac{4}{\pi^2}\sum_{
   1\leq 2j\leq n}\pq{\frac{4(2j)}{(4j+1)(4j-1)}}^2
\]
and
\[\sum_{
   1\leq k=2j+1\leq n}[\operatorname{sinc}(x_k)-\operatorname{sinc}(x_{k-1})]^2=\frac{4}{\pi^2}\sum_{
   1\leq2j+1\leq n}\pq{\frac{4(2j+1)}{(4j+3)(4j+1)}}^2.
\]
Therefore, $v_2(\operatorname{sinc})<\infty$ as a consequence of $\sum_{j=1}^\infty j^{-2}<\infty$.
\end{proof}
The following lemma provides an upper bound on the $L^p$-distance, $p\in[1,\,2)$, between
probability densities with finite absolute moment of (some) order $u>0$,
in terms of the product of the sup-norm distance and any $L^q$-distance, $q>1$. The proof
is similar to that of statement $(b)$ in Lemma~4 by Nguyen~\cite{Ng?}, pages~18 and 24.
\begin{lem}\label{p<2}
Let $f,\,g\in L^\infty(\mathbb{R})$ be probability densities with
$\operatorname{E}_f[|X|^u]<\infty$ and $\operatorname{E}_g[|X|^u]<\infty$ for some real $u>0$.
For every $p\in[1,\,2)$ and $t>0$ such that $pt>1$,
\begin{eqnarray*}
\|f-g\|_p^p\leq(s^{-1}+u)&&\\\nonumber&&\hspace*{-2.1cm}
\times\,[s^{-1/s}(2^{1/s}/u)^u\|f-g\|_{pt}^{pu}\|f-g\|_\infty^{(p-1)/s}
(\operatorname{E}_f[|X|^u]+\operatorname{E}_g[|X|^u])^{1/s}]^{s/(1+su)},
\end{eqnarray*}
where $s^{-1}:=1-t^{-1}$.
\end{lem}
\begin{proof}
For every $R>0$, by H\"{o}lder's inequality,
$\int_{|x|\leq R}|f(x)-g(x)|^p\,\mathrm{d}x\leq (2R)^{1/s}\|f-g\|_{pt}^p$. Also,
$\int_{|x|>R}|f(x)-g(x)|^p\,\mathrm{d}x\leq R^{-u}\|f-g\|_\infty^{p-1}
(\operatorname{E}_f[|X|^u]+\operatorname{E}_g[|X|^u])$.
Thus, $\|f-g\|_p^p\leq\inf_{R>0}[(2R)^{1/s}\|f-g\|_{pt}^p+R^{-u}\|f-g\|_\infty^{p-1}
(\operatorname{E}_f[|X|^u]+\operatorname{E}_g[|X|^u])]$.
The inequality in the assertion follows from $\min_{x>0}(Ax^\alpha+Bx^{-\beta})=(\alpha+\beta)[(A/\beta)^\beta(B/\alpha)^\alpha]^{1/(\alpha+\beta)}$
for every $A,\,B,\,\alpha,\,\beta>0$.
\end{proof}

The next lemma provides an alternative bound on the $L^1$-distance
in terms of the sup-norm distance only.
\begin{lem}\label{L1tosup}
Let $f$ and $g$ be probability densities on $\mathbb{R}$.
For every $\upsilon\in(0,\,1]$ such that
$\int f^\upsilon\,\mathrm{d}\lambda<\infty$, we have
$\|f-g\|_1\leq 2\|f-g\|_\infty^{1-\upsilon}\int f^\upsilon\,\mathrm{d}\lambda$.
\end{lem}
\begin{proof}
Write $\|f-g\|_1=2\int(f-g)^+\,\mathrm{d}\lambda\leq2\int\min\{f,\,\|f-g\|_\infty\}\,\mathrm{d}\lambda
\leq2\|f-g\|_\infty^{1-\upsilon}\int f^\upsilon\,\mathrm{d}\lambda$.
The assertion follows.
\end{proof}
As noted in Remark~3 by Devroye~\cite{Devroye92}, page~2042, if
\begin{equation}\label{eq:moments}
\mbox{for some real $q>0$,}\qquad \operatorname{E}_f[|X|^q]<\infty,
\end{equation}
then $\int f^\upsilon\,\mathrm{d}\lambda<\infty$ for any real $\upsilon\in((1+q)^{-1},\,1)$.
Condition~\eqref{eq:moments} is verified, for example, for a Student's-$t$
distribution with $\nu$ degrees of freedom when $q\in(0,\,\nu)$.
\medskip

The following lemma provides an upper bound on the number
of components of a mixture, whose kernel density belongs to some class $\mathcal{A}^{\rho,\,r,\,L}(\mathbb{R})$,
which uniformly approximates a given compactly supported
mixture with the same kernel.
\begin{lem}\label{Npoints}
Let $K\in\mathcal{A}^{\rho,\,r,\,L}(\mathbb{R})$ for some $\rho,\,r,\,L>0$.
Let $\varepsilon\in(0,\,1)$, $0<a<\infty$ and $\sigma>0$ be given. For any
probability measure $F$ on $[-a,\,a]$, there exists a discrete
probability measure $F'$ on $[-a,\,a]$, with at most
\[
N\lesssim\max\pg{\log(1/\varepsilon),\,(a/\sigma)}, \,\,\,\,\quad\text{if }\,\,\, S_K<\infty,
\]
and
\[
N\lesssim\left\{
\begin{array}{lll}
\log(1/\varepsilon), & \text{if }\,\,\, 0<r<1\,\,\,\text{and}\,\,\,\rho\sigma/a=O((\log(1/\varepsilon))^{(1-r)/r}),\\[2pt]
\log(1/\varepsilon), & \text{if }\,\,\,\,\,\,\,\,\,\,\,\,\,\, r=1\,\,\,\text{and}\,\,\,a/(\rho\sigma)\leq e^{-1},\\[2pt]
\max\pg{\log(1/\varepsilon),\,(a/\sigma)^{r/(r-1)}}, & \text{if }\,\,\,\,\,\,\,\,\,\,\,\,\,\, r>1\,\,\,\text{and}\,\,\,a/(\rho\sigma)\geq e^{-1},
\end{array}\right.
\]
if $S_K=\infty$, support points, such that
$\|{F\ast K_\sigma-F'\ast K_\sigma}\|_\infty\lesssim\varepsilon/\sigma$.
\end{lem}
\begin{proof}
By Lemma~A.1 of Ghosal and van der Vaart~\cite{GvdV01}, page~1260,
there exists a discrete probability measure
$F'$ on $[-a,\,a]$, with at most $N+1$ support points,
$N$ being suitably chosen later on,
such that it matches the moments of $F$
up to the order $N$,
\begin{equation}\label{matching}
\int_{-a}^{a}\theta^j\,
\mathrm{d}F'(\theta)=\int_{-a}^{a}\theta^j\,
\mathrm{d}F(\theta),\qquad j=1,\,\ldots,\,N.
\end{equation}
By the moment matching condition in (\ref{matching}),
\begin{equation}\label{cf}
|\hat{F}(t)-\widehat{F'}(t)|\leq\int_{-a}^{a}
\frac{|t\theta|^N}{N!}\min\pg{\frac{|t\theta|}{N+1},\,2}\,\mathrm{d}(F+F')(\theta)
,\qquad t\in\mathbb{R},
\end{equation}
where the inequality holds because $F$ and $F'$ have finite absolute moments of any order,
see, \emph{e.g.}, inequality (26.5) in Billingsley~\cite{BIL}, page~343.
By the assumption that
$K\in\mathcal{A}^{\rho,\,r,\,L}(\mathbb{R})$, $\int|\hat{K}(\sigma t)|\,\mathrm{d}t<\infty$,
hence $F\ast K_{\sigma}$ and $F'\ast K_{\sigma}$ can be recovered using the inversion formula.
By (\ref{cf}), $\|F\ast K_\sigma-F'\ast K_\sigma\|_\infty\leq 2a^N/(\pi N!)\int
|t|^N|\hat{K}(\sigma t)|\,\mathrm{d}t$.
Next, we distinguish the case where $S_K<\infty$ from the case where $S_K=\infty$.
If $S_K<\infty$, by the assumption that $K\in\mathcal{A}^{\rho,\,r,\,L}(\mathbb{R})$,
\[
\|F\ast K_\sigma-F'\ast K_\sigma\|_\infty
\leq\frac{2}{\pi}\frac{a^N}{N!}\int_{|t|\leq S_K/\sigma}
|t|^N|\hat{K}(\sigma t)|\,\mathrm{d}t
\leq\frac{4}{\sigma}[L+C(\rho,\,r)/\pi]\pt{\frac{aeS_K}{\sigma N}}^N
\lesssim\frac{\varepsilon}{\sigma}
\]
for $N\gtrsim \max\pg{\log(1/\varepsilon),\,(ae^2S_K/\sigma)}$.
If $S_K=\infty$, by the Cauchy-Schwarz inequality,
\[
\begin{split}
\|F\ast K_\sigma-F'\ast K_\sigma\|_\infty
&\leq\frac{2}{\pi}\frac{a^N}{N!}\pt{\frac{2\pi L}{\sigma}}^{1/2}\pt{\int
|t|^{2N}e^{-2(\rho\sigma|t|)^r}\,\mathrm{d}t}^{1/2}\\
&\lesssim\frac{1}{\sigma}\pt{\frac{a}{2^{1/r}\rho\sigma}}^N\frac{[\Gamma((2N+1)/r)]^{1/2}}{\Gamma(N+1)}.
\end{split}
\]
Using $\Gamma(az+b)\sim (2\pi)^{1/2}e^{-az}(az)^{az+b-1/2}$ ($z\rightarrow\infty$ in $|\arg z|<\pi,\,a>0$),
\[\|F\ast K_\sigma-F'\ast K_\sigma\|_\infty\lesssim\frac{1}{\sigma}\pt{\frac{a}{\rho\sigma}}^N
e^{N(1-1/r)}r^{-N/r}N^{-N(1-1/r)+(1/r-3/2)/2}.\]
If $0<r<1$ and
$(\rho\sigma/a)^{r/(1-r)}=O(\log(1/\varepsilon))$, for
\[\pt{\log\frac{1}{\varepsilon}}\lesssim N\lesssim\pt{\frac{\sigma}{a}}^{r/(1-r)},\]
we have
\[\|F\ast K_\sigma-F'\ast K_\sigma\|_\infty
\lesssim\frac{1}{\sigma}
N^{(1/r-3/2)/2}
\exp{\pt{-N\pq{\log\frac{\rho\sigma/a}{N^{1/r-1}}-\pt{1-\frac{1}{r}+\frac{1}{r}\log\frac{1}{r}}}}}
\lesssim\frac{\varepsilon}{\sigma}.
\]
If $r=1$ and $a/(\rho\sigma)\leq e^{-1}$, for $N=\log(1/\varepsilon)$,
\[\|F\ast K_\sigma-F'\ast K_\sigma\|_\infty\lesssim\frac{1}{\sigma}\pt{\frac{a}{\rho\sigma}}^N\lesssim\frac{\varepsilon}{\sigma}.\]
If $r>1$ and $a/(\rho\sigma)\geq e^{-1}$, for
\[N\lesssim\max\pg{\pt{\log\frac{1}{\varepsilon}},\,\pt{\frac{a}{\sigma}}^{r/(r-1)}},\]
we have
\[\|F\ast K_\sigma-F'\ast K_\sigma\|_\infty\lesssim\frac{1}{\sigma}
\exp{\pt{-N\pq{\log\frac{N^{1-1/r}}{a/(\rho\sigma)}-\frac{1}{r}(r-1-\log r)}}}\lesssim\frac{\varepsilon}{\sigma}
\]
and the proof is complete.
\end{proof}
\begin{rmk}\label{compactness}
\emph{Even if stated for a probability measure $F$
supported on a symmetric interval $[-a,\,a]$,
Lemma~\ref{Npoints} holds for every $F$ with $\textrm{supp}(F)$
being \emph{any} compact interval.}
\end{rmk}




The inequality in the next lemma can be proved similarly to the one for
the Gaussian kernel, see, \emph{e.g.}, the first part of
Lemma~1 in Ghosal \emph{et al.}~\cite{GGR}, pages~156--157.

\begin{lem}\label{L1norm}
Let $K$ be a probability density on $\mathbb{R}$, bounded and symmetric around $0$.
For every $\sigma>0$ and every $\theta_j,\,\theta_k\in\mathbb{R}$,
\[
\|{K_{\sigma}(\cdot-\theta_j)-K_{\sigma}(\cdot-\theta_k)}\|_1\leq
2\|K\|_\infty\frac{\abs{\theta_j-\theta_k}}{\sigma}\lesssim
\frac{\abs{\theta_j-\theta_k}}{\sigma}.
\]
\end{lem}

In the following lemma, a sufficient condition is provided for
the $L^1$-distance between kernel mixtures with different
variances to be bounded above by the distance between the variances.
\begin{lem}\label{Ksigma}
Let $K$ be a probability density on $\mathbb{R}$ symmetric around $0$ and monotone decreasing in $|x|$.
For every probability measure $F$ on
$\mathbb{R}$ and every $\sigma,\,\sigma'>0$, we have
$
\|{F*K_{\sigma}-F*K_{\sigma'}}\|_{1}\leq
\|{K_{\sigma}-K_{\sigma'}}\|_1\leq 2
|\sigma-\sigma'|/(\sigma\wedge\sigma')$.
\end{lem}
\begin{proof}
Note that
$
\|{F*K_{\sigma}-F*K_{\sigma'}}\|_{1}\leq
\int\|K_{\sigma}(\cdot-\theta)-K_{\sigma'}(\cdot-\theta)\|_1\,\mathrm{d}F(\theta)
=\|K_{\sigma}-K_{\sigma'}\|_1.
$
The second inequality can be proved as in Norets and Pelenis~\cite{NorPel?},
page~18.
\end{proof}

The next lemma provides an upper bound on
the $L^1$-metric entropy of sets of mixtures with supersmooth kernels.
For $\varepsilon>0$, the metric entropy of a set $B$ in a metric space with metric $d$
is defined as $\log N(\varepsilon,\,B,\,d)$, where $N(\varepsilon,\,B,\,d)$
is the minimum number of balls of radius $\varepsilon$ needed to cover $B$.
The result is based on Lemma~\ref{Npoints}, Lemma~\ref{L1norm},
Lemma~\ref{Ksigma} and can be proved similarly to Lemma~3
of Ghosal and van der Vaart~\cite{GvdV072},
pages~705--707, which deals with normal mixtures.
\begin{lem}\label{L1-entropy}
Let $K\in\mathcal{A}^{\rho,\,r,\,L}(\mathbb{R})$, for some $\rho,\,r,\,L>0$, be a probability density on $\mathbb{R}$
symmetric around $0$ and monotone decreasing in $|x|$.
Let $\varepsilon\in(0,\,1/5)$, $0<s\leq S<\infty$ and $0<a<\infty$ be such that, for some $\nu>0$,
$(a/s)\lesssim(\log(1/\varepsilon))^\nu$.
Define $\F_{a,\,s,\,S}:=\{F*K_{\sigma}:\,F([-a,\,a])=1,\,\,\,s\leq\sigma\leq
S\}$. Then,
\[
\log
N(\varepsilon,\,\mathscr{F}_{a,\,s,\,S},\,\|\cdot\|_1)\,\lesssim\,
\log\pt{\frac{S}{s\varepsilon}}+N
\times\pq{\log\pt{\frac{2a}{s\varepsilon}+1}+\log\frac{1}{\varepsilon}},
\]
\noindent where
\[
N\lesssim\left\{
\begin{array}{lll}
\dfrac{a}{s}\,\times\,\pt{\log\dfrac{1}{\varepsilon}}^{1/r}, & \quad\text{if }\,\,\, 0<r\leq1,\\[10pt]
\max\pg{\pt{\dfrac{a}{s}}^{r/(r-1)},\,\pt{\log\dfrac{1}{\varepsilon}}}, &\quad\text{if }\,\,\,\,\,\,\,\,\quad r>1.
\end{array}\right.
\]
\end{lem}
\medskip

The following lemma is a variant of Lemma~6 in
Ghosal and van der Vaart~\cite{GvdV072}, page~711.
\begin{lem}\label{lem:lowerboundf0}
Let $K$ be a probability density on $\mathbb{R}$ symmetric around $0$.
Let $f$ be a strictly positive and bounded probability density,
non-decreasing on $(-\infty,\,a)$, non-increasing on $(b,\,\infty)$
and such that $f\geq \ell>0$ on $[a,\,b]$. For every $\zeta\in(0,\,1)$, let $\tau_\zeta>0$ be such that $\int_0^{b-a}K_{\tau_\zeta}(x)\,\mathrm{d}x\geq\zeta$. Then, for every $\sigma\in(0,\,\tau_\zeta)$, we have
$f\ast K_\sigma\geq C_\zeta f$, with $C_\zeta:=(\zeta\ell/\|f\|_\infty)\in(0,\,1)$.
\end{lem}

\bibliographystyle{ba}

\begin{thebibliography}{50}
\newcommand{\enquote}[1]{``#1''}
\expandafter\ifx\csname natexlab\endcsname\relax\def\natexlab#1{#1}\fi
\expandafter\ifx\csname url\endcsname\relax
  \def\url#1{{\tt #1}}\fi
\expandafter\ifx\csname urlprefix\endcsname\relax\def\urlprefix{URL }\fi



\bibitem{AS}
Abramowitz, M. and Stegun, I.~A. (1964).
\newblock \enquote{Handbook of Mathematical Functions with Formulas, Graphs, and Mathematical Tables.}
\newblock {\em National Bureau of Standards, Applied Mathematics Series\/}, 55.
\newblock U.S. Government Printing Office, Washington, D.C. Available online at
\AD{http://www.math.sfu.ca/~cbm/aands/}

\bibitem{AL}
Athreya, K.~B. and Lahiri, S.~N. (2006).
\newblock {\em Measure Theory and Probability Theory\/}.
\newblock New York: Springer.


\bibitem{BSW}
Barron, A., Schervish, M.~J. and Wasserman, L. (1999).
\newblock \enquote{The consistency of posterior distributions in nonparametric problems.}
\newblock {\em The Annals of Statistics\/}, 27: 536--561.


\bibitem{BIL}
Billingsley, P. (1995).
\newblock {\em Probability and Measure\/}.
\newblock New York: John Wiley \& Sons, Inc., 3rd edition.


\bibitem{BTI08}
Butucea, C. and Tsybakov, A.~B. (2008).
\newblock \enquote{Sharp optimality in density deconvolution with dominating bias. I.}
\newblock {\em Theory of Probability and Its Applications\/}, 52: 24--39.




\bibitem{Davis}
Davis, K.~B. (1977).
\newblock \enquote{Mean integrated square error properties of density estimates.}
\newblock {\em The Annals of Statistics\/}, 5: 530--535.

\bibitem{deJvZ}
de Jonge, R. and van Zanten, J.~H. (2010).
\newblock \enquote{Adaptive nonparametric Bayesian inference using
location-scale mixture priors.}
\newblock {\em The Annals of Statistics\/}, 38: 3300--3320.

\bibitem{Devroye92}
Devroye, L. (1992).
\newblock \enquote{A note on the usefulness of superkernels in density estimation.}
\newblock {\em The Annals of Statistics\/}, 20: 2037--2056.


\bibitem{F}
Ferguson, T.~S. (1983).
\newblock \enquote{Bayesian density estimation by mixtures of normal
distributions.}
\newblock In Rizvi, M.~H., Rustagi, J.~S. and Siegmund, D.
(eds.), {\em Recent Advances in Statistics\/},  287--302. New York: Academic Press.

\bibitem{G01}
Ghosal, S. (2001).
\newblock \enquote{Convergence rates for density estimation with Bernstein polynomials.}
\newblock {\em The Annals of Statistics\/}, 29: 1264--1280.


\bibitem{GGR}
Ghosal, S., Ghosh, J.~K. and Ramamoorthi, R.~V. (1999).
\newblock \enquote{Posterior consistency of Dirichlet mixtures in density estimation.}
\newblock {\em The Annals of Statistics\/}, 27: 143--158.


\bibitem{GGvdV00}
Ghosal, S., Ghosh, J.~K. and van der Vaart, A.~W. (2000).
\newblock \enquote{Convergence rates of posterior distributions.}
\newblock {\em The Annals of Statistics\/}, 28: 500--531.


\bibitem{GvdV01}
Ghosal, S. and van der Vaart, A.~W. (2001).
\newblock \enquote{Entropies and rates of
convergence for maximum likelihood and Bayes estimation for mixtures
of normal densities.}
\newblock {\em The Annals of Statistics\/}, 29: 1233--1263.

\bibitem{GvdV071}
--- (2007a).
\newblock \enquote{Convergence rates of
posterior distributions for noniid observations.}
\newblock {\em The Annals of Statistics\/}, 35: 192--223.


\bibitem{GvdV072}
--- (2007b).
\newblock \enquote{Posterior convergence rates
of Dirichlet mixtures at smooth densities.}
\newblock {\em The Annals of Statistics\/}, 35: 697--723.

\bibitem{GN11}
Gin\'e, E. and Nickl, R. (2011).
\newblock \enquote{Rates of contraction for posterior distributions in $L^r$-metrics,
$1\leq r\leq\infty$.}
\newblock {\em The Annals of Statistics\/}, 39: 2883--2911.


\bibitem{GL96}
Golubev, G.~K. and Levit, B.~Y. (1996).
\newblock \enquote{Asymptotically efficient estimation for analytic distributions.}
\newblock {\em Mathematical Methods of Statistics\/}, 5: 357--368.

\bibitem{GLT96}
Golubev, G.~K., Levit, B.~Y. and Tsybakov, A.~B. (1996).
\newblock \enquote{Asymptotically efficient estimation of analytic functions in Gaussian noise.}
\newblock {\em Bernoulli\/}, 2: 167--181.

\bibitem{GT98}
Guerre, E. and Tsybakov, A.~B. (1998).
\newblock \enquote{Exact asymptotic minimax constants for the estimation of analytical functions in $L_p$.}
\newblock {\em Probability Theory and Related Fields\/}, 112: 33--51.


\bibitem{HI1990}
Hasminskii, R. and Ibragimov, I. (1990).
\newblock \enquote{On density estimation in the view of Kolmogorov's ideas in approximation theory.}
\newblock {\em The Annals of Statistics\/}, 18: 999--1010.


\bibitem{H95}
Hurst, S. (1995).
\newblock \enquote{The characteristic function of the Student \emph{t} distribution.}
\newblock {\em Financial Mathematics Research Report} No. FMRR006-95,
\newblock {\em Statistics Research Report} No. SRR044-95.



\bibitem{IH83}
Ibragimov, I.~A. and Hasminskii, R.~Z. (1983).
\newblock \enquote{Estimation of distribution density.}
\newblock {\em Journal of Soviet Mathematics\/}, 21: 40--57.



\bibitem{IJ2001}
Ishwaran, H. and James, L.~F. (2001).
\newblock \enquote{Gibbs sampling methods for stick-breaking priors.}
\newblock {\em Journal of the American Statistical Association\/}, 96: 161--173.

\bibitem{IZ2000}
Ishwaran, H. and Zarepour, M. (2000).
\newblock \enquote{Markov chain Monte Carlo in approximate Dirichlet
and beta two-parameter process hierarchical models.}
\newblock {\em Biometrika\/}, 87: 371--390.

\bibitem{K72}
Kawata, T. (1972).
\newblock {\em Fourier Analysis in Probability Theory\/}.
\newblock {\em Probability and Mathematical Statistics\/}, No. 15.
\newblock New York-London: Academic Press.


\bibitem{KRvdV10}
Kruijer, W., Rousseau, J. and van der Vaart, A. (2010).
\newblock \enquote{Adaptive Bayesian density estimation with location-scale mixtures.}
\newblock {\em Electronic Journal of Statistics\/}, 4: 1225--1257.


\bibitem{LMP05}
Lijoi, A., Mena, R.~H. and Prünster, I. (2005).
\newblock \enquote{Hierarchical mixture modeling with normalized inverse-Gaussian priors.}
\newblock {\em Journal of the American Statistical Association\/}, 100: 1278--1291.





\bibitem{L}
Lo, A.~Y. (1984).
\newblock \enquote{On a class of Bayesian nonparametric estimates: I.
Density estimates.}
\newblock {\em The Annals of Statistics\/}, 12: 351--357.




\bibitem{MM11}
Maugis, C. and Michel, B. (2011).
\newblock \enquote{Adaptive density estimation using finite Gaussian mixtures.}
\texttt{arXiv:1103.4253}



\bibitem{Ng?}
Nguyen, X. (2011).
\newblock \enquote{Convergence of latent mixing measures in nonparametric and mixture models.}
\newblock {\em Preprint\/}.


\bibitem{NorPel?}
Norets, A. and Pelenis, J. (2011).
\newblock \enquote{Posterior consistency in conditional density estimation
by covariate dependent mixtures.}
\newblock {\em Preprint\/}.


\bibitem{PY97}
Pitman, J. and Yor, M. (1997).
\newblock \enquote{The two-parameter Poisson-Dirichlet
distribution derived from a stable subordinator.}
\newblock {\em The Annals of Probability\/}, 25: 855--900.













\bibitem{SG}
Shen, W. and Ghosal, S. (2011).
\newblock \enquote{Adaptive Bayesian multivariate density
estimation with Dirichlet mixtures.}
\texttt{arXiv:1109.6406}




\bibitem{Titchmarsh1937}
Titchmarsh, E.~C. (1937).
\newblock {\em Introduction to the Theory of Fourier Integrals}.
\newblock Oxford: The Clarendon Press.



\bibitem{vdV+vZ}
van der Vaart, A.~W. and van Zanten, J.~H. (2009).
\newblock \enquote{Adaptive Bayesian estimation using a Gaussian
random field with inverse Gamma bandwidth.}
\newblock {\em The Annals of Statistics\/}, 37: 2655--2675.








\bibitem{WL63}
Watson, G.~S. and Leadbetter, M.~R. (1963).
\newblock \enquote{On the estimation of the probability density.}
\newblock {\em The Annals of Mathematical Statistics\/}, 34: 480--491.



\bibitem{WS}
Wong, W.~H. and Shen, X. (1995).
\newblock \enquote{Probability inequalities for likelihood ratios and convergence rates of sieve MLEs.}
\newblock {\em The Annals of Statistics\/}, 23: 339--362.







\end{thebibliography}


\begin{acknowledgement}
The author is grateful to Prof. A. W. van der Vaart
for his availability, insightful comments, remarks and suggestions
while visiting the Department of Mathematics of VU University Amsterdam,
whose kind hospitality is gratefully acknowledged.
\end{acknowledgement}

\end{document}